%% file: main.tex
\documentclass[12pt]{amsart}
\usepackage{graphicx} 

\usepackage[T1]{fontenc}
\usepackage[utf8]{inputenc}
\usepackage[english]{babel}
\usepackage{amsmath,amsthm,amssymb,latexsym,amsfonts, mathtools}
\usepackage{mathrsfs} 
\usepackage{tikz-cd}
\usepackage{framed}
\usepackage{listings}
\usepackage{shuffle}
\usepackage{faktor} 
\usepackage{nicefrac}
\usepackage{color,soul}
\usepackage{setspace}
\usepackage{graphicx}
\usepackage{faktor}
\usepackage[a4paper, margin=1in]{geometry}
\usepackage[alphabetic]{amsrefs}
\usepackage[colorlinks,linkcolor=black,anchorcolor=blue,citecolor=blue,hidelinks]{hyperref}

\usepackage{xpatch}
\makeatletter   
\xpatchcmd{\@tocline}
{\hfil\hbox to\@pnumwidth{\@tocpagenum{#7}}\par}
{\ifnum#1<0\hfill\else\dotfill\fi\hbox to\@pnumwidth{\@tocpagenum{#7}}\par}
{}{}
\makeatother

\usepackage{etoolbox}
\makeatletter
\patchcmd\@thm
  {\let\thm@indent\indent}{\let\thm@indent\noindent}%
  {}{}
\makeatother

\makeatletter
\patchcmd{\@maketitle}
  {\ifx\@empty\@dedicatory}
  {\ifx\@empty\@date \else {\vskip3ex \centering\footnotesize\@date\par\vskip1ex}\fi
   \ifx\@empty\@dedicatory}
  {}{}
\patchcmd{\@adminfootnotes}
  {\ifx\@empty\@date\else \@footnotetext{\@setdate}\fi}
  {}{}{}
\makeatother

\makeatletter
\patchcmd{\@thm}{\thm@headfont{\scshape}}{\thm@headfont{\scshape\normalfont\bfseries}}{}{}
\patchcmd{\@thm}{\thm@notefont{\fontseries\mddefault\upshape}}{}{}{}
\makeatother
\usepackage{subfiles}

\DeclareFontFamily{U}{wncy}{}
\DeclareFontShape{U}{wncy}{m}{n}{<->wncyr10}{}
\DeclareSymbolFont{mcy}{U}{wncy}{m}{n}
\DeclareMathSymbol{\sh}{\mathord}{mcy}{"78}

\renewcommand{\epsilon}{\varepsilon}
\renewcommand{\theta}{\vartheta}
\renewcommand{\rho}{\varrho}
\let\temp\phi
\let\phi\varphi
\let\varphi\temp
\newcommand{\insubfile}[1]{\ifx\@onlypreamble\@notprerr\else#1\fi}

\newcommand{\N}{\mathbb{N}}
\newcommand{\Z}{\mathbb{Z}}
\newcommand{\Q}{\mathbb{Q}}
\newcommand{\R}{\mathbb{R}}
\newcommand{\C}{\mathbb{C}}

\newcommand{\modulo}[1]{\;(\textnormal{mod}\; #1)}

\theoremstyle{definition}
\newtheorem{definition}{Definition}
\theoremstyle{plain}
\newtheorem{proposition}[definition]{Proposition}
\theoremstyle{plain}
\newtheorem{theorem}[definition]{Theorem}
\theoremstyle{plain}
\newtheorem{lemma}[definition]{Lemma}
\theoremstyle{plain}
\newtheorem{corollary}[definition]{Corollary}
\theoremstyle{plain}

\theoremstyle{definition}
\newtheorem{lemma*}{Lemma}
\theoremstyle{definition}
\newtheorem{remark}[definition]{Remark}
\theoremstyle{definition}

\AtBeginDocument{%
   \def\MR#1{}
}

\title{Irrationality of values of $L$-functions for real Dirichlet characters}
\author{Veekesh Kumar, Riccardo Tosi}
\address[Veekesh Kumar]{Department of Mathematics, Indian Institute of Technology Dharwad, Chikkamalligawad village, Dharwad, Karnataka - 580011, India.}

\email[]{veekeshk@iitdh.ac.in}
\address[Riccardo Tosi]{Mathematisches Institut, Universität Freiburg, Ernst-Zermelo-Str.~1, 79104, Freiburg im Breisgau, Germany.}

\email[]{riccardo.tosi@math.uni-freiburg.de}

\begin{document}

    \begin{abstract}
        \subfile{0_Abstract}
    \end{abstract}
    \noindent

    \maketitle

    \noindent
    2020 \textit{Mathematics Subject Classification.} Primary  	11J72, Secondary 11M06, 11M35. \\
    \textit{Keywords:} Irrationality; special values of Dirichlet $L$-functions, hypergeometric series.
    
    \tableofcontents

    \addtocontents{toc}{\protect\setcounter{tocdepth}{1}}
    
    \section*{Introduction}

\subfile{1_Introduction}

    \section{Notation}
        \subfile{2_Notation}

    \section{Linear forms}
        \subfile{3_Linear_forms}
    \section{Denominators}
        \subfile{4_Denominators}

    \section{Analytic estimates}

\subfile{5_Analytic_estimates}
    \section{Proof of the main result}
        \subfile{6_Proof_of_the_main_result}

    \bibliographystyle{alpha}
    \bibliography{bibliography.bib}
\end{document}

%% file: 0_Abstract.tex
    Let $\chi$ be a non-trivial real Dirichlet character and $s$ be a sufficiently large integer. In this paper, we prove that among the first $s$ non-critical special values of the Dirichlet $L$-function associated with $\chi$ at least $c\sqrt{\frac{s}{\log s}}$ are irrational, where $c$ is a constant independent of $s$. This is a substantial improvement on work of Fischler when $4$ divides the conductor of $\chi$. Otherwise, our constant improves on work of Mistiaen. \\
    The proof adapts some techniques of Lai and Yu for odd zeta values to the case of $L$-functions. The lack of positivity arguments is overcome by applying an irrationality criterion by Lai and Sprang originating from the $p$-adic setting. 

%% file: 1_Introduction.tex
    For a real Dirichlet character $\chi$ of conductor $T\ge 2$, consider its associated $L$-function
    \[
        L(z,\chi)=\sum_{n=1}^\infty \frac{\chi(n)}{n^z}, \qquad \text{Re}(z)\ge 1.
    \]
    Define $\sigma(\chi)\in\{0,1\}$ by the equation $\chi(-1)=(-1)^{\sigma(\chi)}$. If $j$ is a positive integer with $j\equiv \sigma(\chi) \modulo{2}$, the numbers $L(j,\chi)$ are non-zero algebraic multiples of $\pi^j$, hence, in particular, irrational. We prove the following result on the irrationality of the remaining $L$-values. 
    \begin{theorem}
        \label{theorem: main}
        Let $\chi$ be a real Dirichlet character of conductor $T\ge 2$. For all $\epsilon>0$ there is a positive integer $s_0$ satisfying the following property. For all $s\ge s_0$ of the opposite parity as $\sigma(\chi)$ the set
        \[
            \left\{ L(j,\chi) \mid j\in \Z, \; 2\le j\le s, \; j\equiv \sigma(\chi)+1\modulo{2} \right\}
        \]
        contains at least
        \[
            (c(T)-\epsilon) \sqrt{\frac{s}{\log s}}
        \]
        irrational numbers, where 
        \[
            c(T)=\sqrt{\frac{2\zeta(2)\zeta(3)}{\zeta(6)T}\log \left(\frac{\sqrt{4e^2+T^2}+T}{\sqrt{4e^2+T^2}-T} \right)}.
        \]
    \end{theorem}
    Before discussing Theorem~\ref{theorem: main} in detail, we give a historical overview of the current known results about the irrationality of special values of $L$-functions. Apéry \cite{Apery} proved the irrationality of an odd zeta value for the first time, namely $\zeta(3)$. The subsequent breakthrough \cite{Ball-Rivoal} was Ball and Rivoal's achievement of the linear independence of infinitely many odd zeta values. Their original lower bound on the dimension of the $\Q$-vector space generated by $\zeta(3),\zeta(5),\dots, \zeta(s)$ is 
    \[
        \frac{1+o(1)}{1+\log 2} \log s
    \]
    as $s\to \infty$. Their method was the starting point for several subsequent quantitative
    refinements. \\
    Building on ideas of Zudilin \cite{Zudilin2} and Sprang \cite{Sprang}, the authors of \cite{Fischler-Sprang-Zudilin} showed that there are, asymptotically in $s$, at least 
    \[
        2^{(1-\epsilon)\frac{\log s}{\log \log s}}
    \]
    irrationals among $\zeta(3),\zeta(5),\dots, \zeta(s)$. Their strategy was optimized by Lai and Yu \cite{Lai-Yu} to achieve a lower bound similar to the one in Theorem~\ref{theorem: main}. The constant in the asymptotics has been further refined by Lai \cite{Lai}. With a substantially different approach, Fischler \cite{Fischler-Siegel} upgraded Lai and Yu's result to linear independence, up to a constant factor. \\
    With the exception of \cites{Lai-Yu, Lai}, the results quoted so far have been generalized to the case of $L$-values. Nash \cite{Nash} adapted the strategy of Ball and Rivoal to values of Dirichlet series for $T$-periodic functions whose sum over $[1,T]\cap \Z$ is not zero. Nishimoto \cite{Nishimoto} established the analogue of \cite{Ball-Rivoal} in the case of Dirichlet $L$-functions with the lower bound
    \[
        \frac{1-\epsilon}{T+\log 2} \log s.
    \]
    Using some techniques from \cites{Sprang, Fischler-Sprang-Zudilin}, Fischler \cite{Fischler} established the same asymptotic result as \cite{Fischler-Sprang-Zudilin} in the case of special values of Dirichlet $L$-functions. The method of his subsequent improvement for zeta values \cite{Fischler-Siegel} was later generalized by Mistiaen \cite{Mistiaen} to $L$-values for Dirichlet characters with conductor $T$ not divisible by $4$, with a lower bound on the dimension given by
    \[
        \frac{0.42}{T^{\frac{3}{2}}} \sqrt{\frac{s}{\log s}}.
    \]
    Let us clarify how Theorem~\ref{theorem: main} compares to these latest results. Although the asymptotic bound in $s$ of our result is similar to that of \cite{Mistiaen}, we improve the dependence on the conductor $T$. Indeed, $c(T)$ in Theorem~\ref{theorem: main} dominates the corresponding quantity in Mistiaen's work by a factor $T\sqrt{\log T}$ as $T\to \infty$. \\
    Moreover, Theorem~\ref{theorem: main} applies to any choice of conductor $T$, while the result of \cite{Mistiaen} excludes the case where $4$ divides $T$. In particular, our result is applicable also to the values of the $\beta$-function, associated with the unique non-trivial Dirichlet character of conductor $4$. Note that the best known asymptotic bound on the number of irrational $\beta$-values is the one of \cite{Fischler}. On the other hand, contrary to Mistiaen's work, Theorem~\ref{theorem: main} only deals with real Dirichlet characters and does not provide linear independence. \\
    Finally, the proof method that we adopt for Theorem~\ref{theorem: main} should allow for further refinements taking into account Zudilin's $\Phi_n$-factors \cite{Zudilin}. In this way, one may improve the constant $c(T)$ in Theorem~\ref{theorem: main}, as was observed in \cite{Lai}. Moreover, this could allow for stronger results when specializing to small values of $s$. \\

    Let us now discuss the proof strategy of Theorem~\ref{theorem: main}. Essentially, we adapt the constructions of \cite{Lai-Yu} to the case of Dirichlet $L$-values and overcome some intrinsic difficulties by exploiting and refining some techniques used for irrationality proofs of $p$-adic zeta values \cite{Lai-Sprang}. \\
    A typical strategy to prove the irrationality of at least one among some real numbers $\alpha_1,\dots,\alpha_k$ consists of constructing a sequence of linear forms $L_n(X_1,\dots, X_k)=a_{n,1}X_1+\dots +a_{n,k}X_k$ with integer coefficients such that $L_n(\alpha_1,\dots, \alpha_k)\neq 0$ for all $n$ and at the same time $L_n(\alpha_1,\dots, \alpha_k)\to 0$ as $n\to \infty$. \\
    Ball and Rivoal \cite{Ball-Rivoal} used carefully designed well-poised hypergeometric series to produce such linear forms $L_n$ for $\alpha_i=\zeta(2i+1)$. In order to obtain a result of linear independence, they studied further properties of these linear forms to apply Nesterenko's linear independence criterion \cite{Nesterenko}. The authors of \cite{Fischler-Sprang-Zudilin} modified the construction of Ball and Rivoal to produce several linear forms for all $n$ with related coefficients, then combine them in a careful way to remove many odd zeta values from the resulting linear combinations. However, this technique only provides a lower bound on the number of irrationals among $\zeta(3),\zeta(5),\dots,\zeta(s)$. \\
    By introducing a non-trivial Dirichlet character $\chi$ in the construction of Ball and Rivoal, it is possible to produce linear combinations in the corresponding $L$-values. However, the condition $L_n(\alpha_1,\dots, \alpha_k)\neq 0$, which in the case of zeta values is trivially verified for positivity reasons, becomes now more difficult to handle. Since the sum of $\chi$ over its period equals zero, the leading term in the asymptotic behaviour of $L_n(\alpha_1,\dots, \alpha_k)$ vanishes \cite{Nash}, which forces a study of the sub-leading term. Despite this being possible, it typically requires some sophisticated analytic arguments, such as the saddle point method \cite{Nishimoto} or some zero estimates \cite{Fischler}. \\
    In the way of adapting the construction of \cite{Lai-Yu} to special values of Dirichlet $L$-functions, we propose a more elementary strategy to address this problem. Another context in which a similar difficulty arises is the $p$-adic world. For example, it is not known whether all $p$-adic odd zeta values are non-zero \cite{Calegari}, which makes the task of checking the non-vanishing of their linear combinations even more elusive. \\
    To solve this issue in the $p$-adic setting, Lai and Sprang \cite{Lai-Sprang} exploit a new irrationality criterion. Instead of proving that $L_n(\alpha_1,\dots, \alpha_k)\neq 0$, they observe that it suffices to find an unbounded sequence of primes $p_n$ for which the $p_n$-adic valuation of the coefficient $a_{n,1}$ is strictly smaller than that of $a_{n,i}$ for all $i\ge 2$ and all $n$. They construct linear forms $L_n(X_1,\dots, X_k)=b_{n,1}X_1+\dots +b_{n,k}X_k$ with rational coefficients $b_{n,i}$, following \cite{Lai-Yu}, and note that, for $i\ge 2$,
    \[
        \text{lcm}(1,\dots, n+1)^{s+1}b_{n,1}\in \Z, \qquad \text{lcm}(1,\dots, n)^{s+1-i}b_{n,i}\in \Z.
    \]
    By choosing an infinite sequence of $n\in \N$ such that $n+1$ is prime, they then check that the numerator of $b_{n,1}$ does not clear all the $s+1$ factors of $n+1$ in the denominator. Hence, they may apply their irrationality criterion with the choice $p_n=n+1$. \\
    We adopt a similar strategy for Theorem~\ref{theorem: main} with one crucial technical difference. When adapting the linear forms of \cite{Lai-Yu} to $L$-values, there appears no denominator $n+1$ in any of the coefficients. In particular, for every prime $p\le n$ the difference between the $p$-adic valuation of the denominators of $b_{n,1}$ and $b_{n,i}$ is bounded by $1$, and not by $s+1$. This forces us to choose a completely different sequence of primes $p_n$ for the application of \cite{Lai-Sprang}. This leads us to estimate the $p_n$-adic valuation of the coefficients of these linear forms precisely, rather than giving an upper bound.  To the best of our knowledge, this is the first time that the complex version of Lai and Sprang's criterion is used in an irrationality proof. \\

    The structure of the article is as follows. In Section 1, we set up all necessary parameters for our construction. We also describe the choice for the sequence of primes $p_n$ necessary for the irrationality criterion of \cite{Lai-Sprang}. In Section 2, we introduce our linear forms, which are modelled after \cite{Lai-Yu}. \\
    In Section 3, we estimate precisely the $p_n$-adic valuation of the coefficients of the linear forms. Since this constitutes the novel technical heart of the proof, we will discuss it in detail. In Section 4 and 5, we closely follow \cites{Fischler-Sprang-Zudilin, Lai-Yu} to show that our linear combinations of $L$-values are sufficiently small to conclude the proof of Theorem~\ref{theorem: main}. \\

    \noindent
    \textbf{Acknowledgements.} The authors thank Johannes Sprang for his encouragement and for giving them the opportunity to meet at the University of Duisburg-Essen. They also thank Li Lai for his helpful comments. The research of the first named author research is supported by ANRF-PMECRG grant ANRF/ECRG/2024/002315/PMS and partially supported by NBHM research grant. The second named author acknowledges partial financial support from the DFG Research Training Group 2553.

%% file: 2_Notation.tex
    Let $T\ge 2$ be an integer and $\chi$ be a Dirichlet character of conductor $T$. We assume $\chi$ to be real, so it takes values only in $\{ -1,0,1\}$. Let $\sigma(\chi)\in\{0,1\}$ be defined by the equation
    \[
        \chi(-1)=(-1)^{\sigma(\chi)}.
    \]
    We fix a positive integer $s$ such that $s\equiv \sigma(\chi)+1 \modulo{2}$. Take a positive rational number $r\in \Q$, which we write as $r= \frac{r_1}{r_2}$ with $r_1$ and $r_2$ relatively prime positive integers. \\
    Also fix a positive integer $B>0$, that will eventually be chosen depending on $s$. Following \cite{Lai-Yu}, we define a \emph{denominator set}
    \[
        \Psi_B =\{ b\in \N \mid \phi(b) \le B \}, 
    \]
    where $\phi$ is Euler's totient function, and a \emph{zero set}
    \[
        \mathcal{Z}_B=\left\{ \frac{a}{b}\in \Q \,\middle\vert\, b\in \Psi_B, \, 1\le a \le b, \, \gcd(a,b)=1 \right\}. 
    \]
    Here are some known properties of $\Psi_B$ and $\mathcal{Z}_B$.
    \begin{lemma}[{\cite[Proposition 2.2]{Lai-Yu}}] 
    The following properties hold:
    \begin{enumerate}
        \item For $B\to \infty$ we have
        \[
            |\Psi_B|=\left( \frac{\zeta(2)\zeta(3)}{\zeta(6)}+o(1)\right) B.
        \]
        \item For all $b\in \Psi_B$ we have $\frac{1}{b},\frac{2}{b},\dots, \frac{b}{b}\in \mathcal{Z}_B$.
        \item For all sufficiently large $B$ we have $|\mathcal{Z}_B|\le B^2$.
    \end{enumerate}
    \end{lemma}
    We define
    \[
        P_{B}=2\,T\,r_2 \, \textnormal{lcm}(\,p-1 \mid \textnormal{$p$ is a prime factor of some $b\in \Psi_B$}\,)\,\textnormal{lcm}(\,b\mid b\in \Psi_B\,). 
    \]
    We will also assume that
    \[
        s> (2r+T)|\mathcal{Z}_B|. 
    \]

    \bigskip

    Let us state the main irrationality criterion that will be used in the proof of Theorem~\ref{theorem: main}. It is due to Lai and Sprang \cite[Lemma 2.3]{Lai-Sprang}. For a prime number $p$ we denote by $v_p$ the standard $p$-adic valuation of a rational number.
    \begin{lemma}
    \label{lemma: irrationality criterion}
        Let $\xi_0,\dots, \xi_s$ be real numbers with $\xi_0\neq 0$. Let 
        \[
            L_n=c_{n,0}X_0+\dots + c_{n,s}X_s \in \Z[X_1,\dots, X_s] 
        \]
        be a sequence of non-trivial linear forms. Assume that there is an unbounded subset $\mathcal{N}\subseteq \N$ and, for all $n\in \mathcal{N}$, a prime $p_n$ such that the following conditions hold:
        \begin{enumerate}
            \item $L_n(\xi_0,\dots, \xi_s)\to 0$ as $n\to \infty$; \\
            \item we have $v_{p_n}(c_{n,0})< v_{p_n}(c_{n,i})$ for $i=1,\dots, s$; \\
            \item $p_n\to \infty$ as $n\to \infty$. 
        \end{enumerate}
        Then at least one of $\xi_0,\dots, \xi_s$ is irrational.
    \end{lemma}
    We now specify a set $\mathcal{N}\subseteq \N$ and a sequence of primes $p_n$ to apply Lemma~\ref{lemma: irrationality criterion} in our context. In Remark~\ref{remark: n+1 vs p_n} we will explain why we need a more elaborate choice of $p_n$ than in \cite{Lai-Sprang}. 
    \begin{remark}
        Let us give an intuitive description of the auxiliary primes $p_n$ that we are looking for. In the situation of Lemma~\ref{lemma: irrationality criterion}, we will estimate the first digit, say $d_{n,i}$, in the $p_n$-adic expansion of the coefficients $c_{n,i}$. \\
        Since $c_{n,i}$ will strongly depend on $n$, it is convenient to choose $p_n$ so that the residue class of $n$ modulo $p_n$ is small, in the sense that $n-p_n$ should be small. At the same time, to handle binomial coefficients modulo $p_n$, the $p_n$-adic expansion of $n$ should contain very few digits, so we require $\left\lfloor n/p_n \right\rfloor$ to be small. \\
        At this point, one might be tempted to choose $n-p_n$ and $\left\lfloor n/p_n \right\rfloor$ to be constant in $n$. For example, $p_n=n-1$. However, this would remove the dependence on $n$ in the $p_n$-adic digit $d_{n,i}$, which would depend only on $s$, $B$ and $\mathcal{Z}_B$. Due to the sophisticated definition of $\mathcal{Z}_B$, proving $d_{n,i}\neq 0$ would rely on non-trivial asymptotic estimates in $s$, which would require the saddle point method or discrepancy arguments. \\
        To keep the proof as elementary as possible, we ensure that $n-p_n$ and $\left\lfloor n/p_n \right\rfloor$ tend to infinity as $n\to \infty$ in a very controlled way. This provides us with two more parameters that eventually simplify the task of showing that $d_{n,i}\neq 0$.
    \end{remark}

    Fix $0<\delta<1$. Let $\{\mu_n\}_n\subseteq \N_{\ge 1}$ be an increasing sequence such that also $\frac{n}{\mu_n}$ is strictly increasing. Choose $f,g\colon \N \to \R$ such that
    \begin{align*}
        \mu_n & =o(\sqrt{n}), & \lim_{n\to\infty}f(n) & =+\infty, \\
        f&=o\left(\frac{n}{\mu_n}\right), & \lim_{n\to\infty}\frac{f}{g} &=\delta,
    \end{align*}
    where the $o$-notation is intended for $n\to \infty$. Consider
    \[
        L(k)=\frac{kP_B}{r_2\mu_ {kP_B}}-g(kP_B)-f(kP_B), \qquad U(k)= \frac{kP_B}{r_2\mu_ {kP_B}}-g(kP_B), \qquad I_k=[L(k),U(k)].
    \]
    Owing to the asymptotic properties of $f$ and $g$, it follows easily that $\min I_{k+1}<\max I_k$ for all $k$ larger than some $k_0>0$. Since $\lim_{n\to \infty} U(n)=\infty$, we have
    \[
        \bigcup_{k=k_0}^{\infty} I_k = [L(k_0),+\infty) 
    \]
    By Dirichlet's theorem on primes in arithmetic progressions, there are infinitely many primes congruent to $-1$ modulo $P_B$. Any sufficiently large such prime $p$ belongs to $I_{k_p}$ for some $k_p\ge k_0$. If we set $n_p=k_pP_B$, then
    \[
        p\in \left[ \frac{n_p}{r_2\mu_{n_p}}-g\left( n_p\right) - f\left(n_p\right), \frac{n_p}{r_2\mu_ {n_p}}-g\left( n_p\right) \right].
    \]
    Therefore, there is an infinite set $\mathcal{N}$ of positive integers $n$ such that
    \begin{enumerate}
        \item $P_B$ divides $n$;
        \item there is a prime $p_n\equiv -1\modulo{P_B}$ lying in the interval
        \[
            \left[ \frac{n}{r_2\mu_n}-g\left(n\right) - f\left( n\right), \frac{n}{r_2\mu_n}-g\left(n\right) \right].
        \]
    \end{enumerate}
    We will always assume that $n\in \mathcal{N}$ and write $p_n$ for the corresponding prime. In the following lemma we collect all the properties of these primes that will be used in the proof of Theorem~\ref{theorem: main}. 
    \begin{lemma}
    \label{lemma: properties of auxiliary primes}
        The following properties hold for all $n\in \mathcal{N}$ sufficiently large:
        \begin{enumerate}
            \item The only multiples of $p_n$ that lie in $[1,n]$ are $p_n,2p_n,\dots, r_2\mu_np_n$.
            \item We have 
            \[
                v_{p_n}\left( \left( \frac{n}{r_2}\right) ! \right)=\mu_n, \qquad v_{p_n}(n!)=r_2\mu_n, \qquad v_{p_n}(d_n)=1. 
            \]
            \item For $n\to \infty$,
            \[
                \qquad p_n\sim \frac{n}{r_2\mu_n}, \qquad n-r_2\mu_n=o(n).
            \]
            \item We have
            \[
                    v_{p_n} \left(\binom{n}{h}\right)=0 
            \]
            if and only if
            \[
                h=kp_n+x, \qquad 0\le k\le r_2\mu_n, \; 0\le x\le n-r_2\mu_np_n. 
            \]
        \end{enumerate}
    \end{lemma}
    \begin{proof}
        Since $p_n\ge \frac{n}{r_2\mu_n}-g\left( n\right) - f\left(n\right)$, we have
        \[
            (r_2\mu_n+1)p_n -n\ge \frac{(r_2\mu_n+1)n}{r_2\mu_n}-n-(r_2\mu_n+1)g\left( n\right) - (r_2\mu_n+1)f\left(n\right) \sim \frac{n}{r_2\mu_n},
        \]
        because $f=o\left( \frac{n}{\mu_n} \right)$ and $g=\delta f+o(1)$. This difference is positive for $n$ large enough, so $(r_2\mu_n+1)p_n\not\in [1,n]$. The second and third point follow immediately. \\
        For the fourth claim, let us write
        \[
            n=(n-r_2\mu_np_n)+r_2\mu_np_n, \qquad h=x+kp_n
        \]
        for the $p_n$-adic expansions of $n$ and $p_n$, with $k\in \{0,\dots, r_2\mu_n\}$, $x\in \{0,\dots, p_n-1\}$. Note that the expansion of $n$ only has two digits, because $r_2\mu_n p_n<p_n^2$ thanks to the assumption $\mu_n=o(\sqrt{n})$. By Lucas's theorem, we have
        \[
            \binom{n}{h}\equiv \binom{r_2\mu_n}{k}\binom{n-r_2\mu_np_n}{x} \modulo{p_n}.
        \]
        The right-hand side vanishes if and only if $x>n-r_2\mu_np_n$. Hence, the $p_n$-adic valuation of this binomial coefficient is zero if and only if $x\le n-r_2\mu_np_n$.
    \end{proof}
    Let us now make explicit choices for $\delta$, $\mu_n$, $f$ and $g$. We set
    \[
        \mu_n=\left\lfloor \log\log(n) \right\rfloor, \qquad f(n)=\log\log\log(n), \qquad g(n)=\delta f(n).
    \]
    Since $T\ge 2$, we may choose $0<\delta <1$ such that
    \[
       \frac{1}{r+T-1} <\delta <1.
    \]
    Note that 
    \begin{equation}
        \label{eq: estimate for difference n and prime}
        r_2\mu_ng(n)\le n-r_2\mu_np_n \le r_2\mu_n(f(n)+g(n))\le r_2\mu_n g(n) \left(1+\frac{1}{\delta}+o(1)\right).
    \end{equation}
    for all $n\in \mathcal{N}$ sufficiently large. \\

%% file: 3_Linear_forms.tex
    \label{section: linear forms}
    For all $n\in \mathcal{N}$ we consider the  following rational function, modelled after \cite{Lai-Yu}:
    \[
        R_n(t) = C_n \cdot \frac{(t-rn)\prod_{\theta\in \mathcal{Z}_B} \prod_{k=0}^{(2r+T)n-1} (t-rn+k+\theta)}{\prod_{k=0}^n (t+Tk)^{s+1}},
    \]
    where $C_n$ is some rational constant depending only on $n$, $r$, $s$ and $B$. We first describe some structural properties of $R_n(t)$, then give the precise definition of $C_n$ needed for Theorem~\ref{theorem: main}.
    \begin{lemma}
        $R_n(t)$ satisfies the following properties:
        \begin{enumerate}
            \item $\deg R_n(t)\le -2$;
            \item $R_n(-t-Tn)=(-1)^{\sigma(\chi)+1}R_n(t)$;
            \item $R_n(t)$ has a partial fraction decomposition of the form
            \[
                R_n(t)=\sum_{i=1}^s\sum_{k=0}^n\frac{a_{i,k}}{(t+Tk)^i}
            \]
            with $a_{i,n-k}=(-1)^{i+\sigma(\chi)+1} \,a_{i,k}$ for all $i=1,\dots, s$, $k=0,\dots, n$. Moreover,
            \[
                \sum_{k=0}^n a_{1,k}=0.
            \]
        \end{enumerate}
    \end{lemma}
    \begin{proof} 
        The first claim follows from the inequality $ s > (2r+T)|\mathcal{Z}_B|$. The second claim uses that $n\in \mathcal{N}$ is even and $s\equiv \sigma(\chi)+1 \modulo{2}$.
    \end{proof}  
    We now use this rational function to construct rational linear combinations of the desired $L$-values. Given positive integers $h$ and $b$ we define
    \[
        \chi_{h,b}(t)=\chi(bt+h).
    \]
    Note that $\chi_{h,b}$ is $T$-periodic. Moreover, if $\gcd(h,b)=1$, we write $\chi_{\frac{h}{b}}$ instead of $\chi_{h,b}$. In general, we have 
    \[
        \chi_{h,b}=\chi(\gcd(h,b))\chi_{\frac{h}{b}}.
    \]
    For all $\theta\in \mathcal{Z}_B$, we define 
    \[
        r_{n,\theta}=\sum_{t=1}^\infty R_n(t+\theta)\chi_\theta(t).
    \]
     \begin{lemma}
     \label{lemma: coefficients of linear forms}
        For each  $\theta\in \mathcal{Z}_B$, we have 
        \[
            r_{n,\theta}=\rho_{0,\theta}+\sum_{j=2}^s \rho_j \,\zeta(j, \theta; \chi_\theta), \qquad \zeta(j, \theta; \chi_\theta)=\sum_{m=0}^\infty \frac{\chi_\theta(m)}{(m+\theta)^j},
        \]
        for some $\rho_{0,\theta}, \rho_j \in \Q $. For all $j=2,\dots, s$ the coefficient 
        $$
    \rho_j=\sum_{k=0}^na_{j,k}
        $$
        does not depend on $\theta\in\mathcal{Z}_B$, and 
        $$
    \rho_{0,\theta}=-\sum_{k=0}^n\sum_{\ell=0}^{Tk}\sum_{j=1}^s\frac{a_{j,k} \, \chi_\theta(l)}{(\ell+\theta)^j}.
        $$
        Moreover, for all $j=1,\dots, s$ such that $j\equiv \sigma(\chi)\modulo{2}$ we have $\rho_j=0$.
    \end{lemma}
    \begin{proof}
        The proof is essentially \cite[Lemma 1]{Fischler-Sprang-Zudilin}.
    \end{proof}
    \begin{remark}
    If we take positive integers $h$ and $b$, we can also define $r_{n,h,b}$ by replacing $\chi_\theta$ by $\chi_{h,b}$. The previous statement still holds with $\chi_{\theta}$ replaced by $\chi_{h,b}$. We observe that, if $\theta=\frac{h}{b}$,
    \[
        r_{n,h,b}=\chi(\gcd(h,b))r_{n,\theta} \qquad \text{and} \qquad \zeta(j,\theta;\chi_{h,b})=\chi(\gcd(h,b))\zeta(j,\theta;\chi_{\theta}).
    \]
    Thus, the coefficients $\rho_j$ for $j=2,\dots, s$ are the same in $r_{n,h,b}$ and $r_{n,\theta}$.
    \end{remark}
    Next, for all $b\in \Psi_B$ we consider the linear forms
        \[
            r_{n,b}'= \sum_{h=1}^b r_{n,h,b}=\sum_{h=1}^b \chi(\gcd(h,b))\,r_{n,\frac{h}{b}}. 
        \]
    This step allows the desired $L$-values to appear, as shown by the following lemma.
    \begin{lemma}
        For all $b\in B$, we have 
        \[
            r_{n,b}'=\rho_{0,b}' + \sum_{j=2}^s b^j\rho_j \, L(j,\chi),
        \]
        where 
        \[
            \rho_{0,b}'=\sum_{h=1}^b \rho_{0,h,b}= \sum_{h=1}^b\chi(\gcd(h,b))\rho_{0,\frac{h}{b}}.
        \]
    \end{lemma}
    \begin{proof}
        We have 
        \[
            \sum_{h=1}^b \zeta\left(j,\frac{h}{b};\chi_{h,b}\right)= \sum_{h=1}^b\sum_{\ell =0}^\infty \frac{\chi_{h,b}(\ell) }{\left( \ell + \frac{h}{b}\right)^j} = b^j\sum_{h=1}^b\sum_{\ell =0}^\infty \frac{\chi(b\ell + h)}{(b\ell+h)^j}= b^j\sum_{n=1}^\infty \frac{\chi(n)}{n^j}=b^j\, L(j,\chi).
        \]
        The statement follows at once.
    \end{proof}
    Finally, for all $b\in \Psi_B$ we will define suitable integers $w_b$ and consider
    \[
        r_n''=\sum_{b\in \Psi_B} w_b r_{n,b}'=\rho_0''+\sum_{j=2}^s \rho_j''L(j,\chi),
    \]
    which has coefficients
    \[
        \rho_0''=\sum_{b\in \Psi_B} w_b \rho_{0,b}' \qquad \text{and} \qquad \rho_{j}''=\rho_j\left(\sum_{b\in \Psi_B} w_b b^j\right).
    \]
    The choice of the $w_b$'s will be a crucial ingredient for the proof of Theorem~\ref{theorem: main}. Following the elimination procedure of \cite{Fischler-Sprang-Zudilin}, we will ensure that several coefficients $\rho_j''$ vanish. \\

    Finally, let us return to the constant $C_n$ appearing the the definition of $R_n(t)$. We take 
    \[
        C_n=A_1(B)^n A_2(B)^n \frac{(n+1)!^{(s+1)}}{(n/\mathrm{den}(r)!)^{\mathrm{den}(r)(2r+T)|\mathcal{Z}|_B}},
    \]
    where
    $$
        A_1(B)=\prod_{b\in\Psi_B}b^{(2r+T)\phi(b)}\qquad\text{and} \qquad A_2(B)=\prod_{b\in\Psi_B}\prod_{p|b} p^{(2r+T)\phi(b)/(p-1)}.
    $$
    The following result follows immediately from \cite[Lemma 2.4]{Lai-Yu}.
    \begin{lemma}\label{estimate:A_1(B)-A_2(B)}
        As $B\to \infty$, we have
        $$
        A_1(B)=\exp\left(\left(\frac{1}{2}\frac{\zeta(2)\zeta(3)}{\zeta(6)}+o_{B\to\infty}(1)\right)(2r+T)B^2\log B\right),
        $$
        and, for any integer $B$ larger than some absolute constant,
        $$
        A_2(B)\leq \exp(10(2r+T)B^2(\log\log B)^2).
        $$
        \end{lemma}

%% file: 4_Denominators.tex
    In this section we study the denominators of the coefficients of the linear forms $r_{n,\theta}$, $r_{n,b}'$ and $r_{n}''$ introduced before. In order to apply the irrationality criterion of Lemma~\ref{lemma: irrationality criterion}, we also estimate precisely the $p_n$-adic valuation of their coefficients. 
    We first provide suitable upper bounds for the denominators of $a_{i,h}$, $\rho_i$ and $\rho_{0,\theta}$. Then, we give conditional upper bounds on the $p_n$-adic valuation of $a_{i,h}$ in Proposition~\ref{proposition: p-adic valuation of a_i,h}, of $\rho_i$ in Proposition~\ref{proposition: p-adic valuation of rho_i} and of $\rho_{0,\theta}$ in Proposition~\ref{proposition: p-adic valuation of rho_0,theta}.
    These statements will depend on the non-vanishing of certain technical quantities that we specify each time. After collecting all these conditions, in the last subsection we prove the non-vanishing of the quantities necessary to compute the $p_n$-adic valuation of $\rho_{0}''$. This is the content of Proposition~\ref{proposition: det ist not zero}. \\
    
    We always consider $n\in \mathcal{N}$. We set $d_n=\text{lcm}(1,\dots, n)$ and $\delta_n=Td_n$ throughout. Moreover, for a power series $F\in \C[[t]]$ we denote by $\Delta_\ell F$ its $\ell$-th coefficient. 

    \medskip 
    
    \subsection{Upper bound for the denominator of $a_{i,h}$ and $\rho_i$} $ $ \\
    
    Take $\theta=\frac{a}{b}\in \mathcal{Z}_B$ with $\gcd(a,b)=1$. We define
    \[
        \kappa^{(n)}_b =\frac{ \prod_{p\mid b}p^{\frac{(2r+T)n}{p-1}} }{\left(\frac{n}{\textnormal{den}(r)}\right)!^{\textnormal{den}(r)(2r+T)}},
    \]
    as well as the polynomial
    \begin{align*}
        F_{b,a}(t)=\kappa_b^{(n)}\cdot \prod_{j=0}^{(2r+T)n-1} (bt-brn+a+bj).
    \end{align*}
    Further, we set
    \[
        \widetilde{F}_{b,a}(t)=
        \begin{cases}
            F_{b,a}(t) & \text{if $\theta\neq 1$}; \\
            (t-rn)F_{b,a}(t) & \text{if $\theta=1$}.
        \end{cases}
    \]
    Note that for all $h\in\{0,\dots, n\}$, similarly to \cite{Lai-Yu},
    \begin{equation}
    \label{eq: explicit expression for a_i,h}
        a_{i,h}=\binom{n}{h}^{s+1}\sum_{\underline{\ell}} \prod_{\frac{a}{b}\in \mathcal{Z}_B} \Delta_{\ell_{b,a}}\left( \widetilde{F}_{b,a}(t-h)\right) \prod_{\substack{k=0  \\ k\neq h}}^n \frac{(-1)^{\ell_k}\binom{s+\ell_k}{\ell_k}}{T^{\ell_k}(k-h)^{\ell_k}},
    \end{equation}
    where the sum ranges over all $\underline{\ell}= ((\ell_{b,a})_{b,a},(\ell_k)_k)$ such that 
    \[
        \sum_{b,a} \ell_{b,a}+\sum_{\substack{k=0 \\ k\neq h}}^n \ell_k=s+1-i.
    \]
    \begin{lemma}
        The number $d_n^{\ell}$ is a denominator for $\Delta_\ell(F_{b,a}(t-h))$. Moreover, $\delta_n^{s+1-i}$ is a denominator for $a_{i,h}$.
    \end{lemma}
    The proof of the first assertion is exactly as in \cite[Proposition 3.2]{Lai-Yu}. The second assertion follows as in \cite[Lemma 3.3]{Lai-Yu}; we only replace $d_n$ by $\delta_n=Td_n$ to take care of the powers of $T$ in the very last factors of the above expression for $a_{i,h}$. 
    \begin{corollary}
        The number $\delta_n^{s+1-i}$ is a denominator for $\rho_i$.
    \end{corollary}
    \begin{proof}
        The claim follow from the previous lemma together with Lemma~\ref{lemma: coefficients of linear forms}.
    \end{proof}

    \medskip
    
    \subsection{Upper bound for the denominator of $\rho_{0,\theta}$} $ $\\
    
    Recall that
    \[
        \rho_{0,\theta}=-\sum_{k=0}^n\sum_{\ell=0}^{Tk}\sum_{i=1}^s\frac{a_{i,k} \, \chi_\theta(\ell)}{(\ell+\theta)^i}= 
                        -\sum_{k=0}^n\sum_{\ell=0}^{Tk}\chi_\theta(\ell)\left( \sum_{i=1}^s\frac{a_{i,k}}{(\ell+\theta)^i}\right).
    \]
    We want to estimate the denominator of the internal sum, which we denote by
    \[
        \tau_{\theta,k,\ell}=\sum_{i=1}^s\frac{a_{i,k}}{(\ell+\theta)^i}.
    \]
    We have seen that 
    \[
        \delta_n^{s+1-i} a_{i,k} \in \Z.
    \]
    \begin{lemma}
        Fix $\theta\in \mathcal{Z}_B$, $0\le k_0\le n$ and $0\le \ell_0\le Tk_0$. Suppose that at least one of the following conditions holds:
        \begin{enumerate}
            \item $\theta\neq 1$; 
            \item $\theta=1$ and $ \ell_0+1\not\equiv 0 \modulo{T}$.
        \end{enumerate}
        Then $\delta_n^{s+1}\tau_{\theta,k_0,\ell_0}\in \Z$.
    \end{lemma}
    \begin{proof}
        We follow Sprang's arithmetic lemma in \cite{Sprang}. Observe that $R_n(t)$ has zeros at
         \[
            1-\eta, \, -\eta, \, -\eta-1, \,  \dots,\,  -\eta-Tn+2, \, -\eta-Tn+1
         \]
         for all $\eta\in \mathcal{Z}_B\setminus\{1\}$. Since $\mathcal{Z}_B\setminus\{1\}$ is stable under $\eta\mapsto 1-\eta$, it follows that $R_n(t)$ has zeros at
         \[
            \eta,\,  \eta-1,\, \eta-2,\, \dots, \, \eta-1-Tn, \, \eta-Tn
         \]
         for all $\eta\in\mathcal{Z}_B\setminus\{1\}$. Since $-Tn\le \ell_0-Tk_0 \le 0$, it follows that $R_n(\ell_0-Tk_0+\theta)=0$ whenever $\theta\neq 1$. If $\theta=1$, but $\ell_0+1\not\equiv 0\modulo{T}$, then $\ell_0-Tk_0+\theta$ is still a zero of $R_n(t)$, because it does not belong to $\{0,-T,\dots, -nT\}$. \\
         The partial fraction decomposition of $R_n(t)$ therefore implies that 
         \begin{gather*}
            \sum_{i=1}^s \sum_{k=0}^n \frac{a_{i,k}}{(\ell_0-Tk_0+\theta+Tk)^i}=0, \qquad \textnormal {that is,} \\ 
            \tau_{\theta,k_0,\ell_0}=\sum_{i=1}^s\frac{a_{i,k_0}}{(\ell_0+\theta)^i}= \sum_{i=1}^s\sum_{\substack{k=0 \\ k\neq k_0}}^n \frac{a_{i,k}}{(\ell_0+T(k-k_0)+\theta)^i}.
         \end{gather*}
         Assume that the latter quantity multiplied by $\delta_n^{s+1}$ is not an integer. Then, there is a prime $p$ for which
         \[
            v_p \left( \sum_{i=1}^s\frac{\delta_n^{s+1}a_{i,k_0}}{(\ell_0+\theta)^i} \right) <0, \qquad \textnormal{and} \qquad v_p \left( \sum_{i=1}^s\sum_{\substack{k=0 \\ k\neq k_0}}^n \frac{\delta_n^{s+1}a_{i,k}}{(\ell_0+T(k-k_0)+\theta)^i} \right)<0.
         \]
         In particular, at least one summand in both terms, say the $i_0$-summand (resp. $i_1$, $k_1$-summand) in the first (resp. second) term, must have negative $p$-adic valuation. Hence,
         \[
            v_p \left( \frac{\delta_n^{i_0}}{(\ell_0+\theta)^{i_0}} \right) <0, \qquad \textnormal{and} \qquad v_p \left( \frac{\delta_n^{i_1}}{(\ell_0+T(k_1-k_0)+\theta)^{i_1}} \right)<0.
         \]
         In particular, we have 
        \[
            v_p(\delta_n)<v_p(\ell_0+\theta), \qquad v_p(\delta_n)<v_p(\ell_0+T(k_1-k_0)+\theta).
        \]
        This implies that
        \begin{align*}
            v_p(T(k_1-k_0))&=v_p((\ell_0+T(k_1-k_0)+\theta)-(\ell_0+\theta))\\
                            &\ge \min(v_p(\ell_0+T(k_1-k_0)+\theta),v_p(\ell_0+\theta))>v_p(\delta_n).
        \end{align*}
        However, $0\le|k_1-k_0|\le n$, hence $T(k_1-k_0)$ divides $\delta_n$, a contradiction. 
    \end{proof}
    \begin{lemma}
         For $\theta=1$, $0\le k_0\le n$, $0\le \ell_0\le Tk_0$ and $\ell_0+1\equiv 0 \modulo{T}$, we have 
         \[
            \delta_n^{s+1} \tau_{1,k_0,\ell_0}\in \Z.
         \]
     \end{lemma}
     \begin{proof}
         Write $\ell_0+1=Tx$ for some $x\in \Z$. Since $0\le \ell_0\le Tk_0$, we have $1\le \ell_0+1\le Tk_0+1$, so $Tx\le Tk_0+1$. Hence,
        \[
            x\le k_0+\frac{1}{T}\le n+\frac{1}{T}.
        \]
        For $T\ge 2$, we have $x\le n$. Then 
         \[
            \delta_n^{s+1}\tau_{1,k_0,\ell_0}=\sum_{i=1}^s\frac{\delta_n^{s+1}a_{i,k_0}}{(\ell_0+1)^i}=\sum_{i=1}^s\delta_n^{s+1-i}a_{i,k_0}\left(\frac{\delta_n}{Tx}\right)^i\in \Z.
         \]
     \end{proof}
     
     Overall, we have proved:
     \begin{corollary}
         We have $\delta_{n}^{s+1}\rho_{0,\theta}\in \Z$. In particular, also $\delta_{n}^{s+1}\rho_{0,b}'$ and $\delta_{n}^{s+1}\rho_{0}''$ are integers. 
     \end{corollary}

    \begin{remark}
     \label{remark: n+1 vs p_n}
        This corollary shows a crucial difference between our situation and the one of zeta values. Indeed, in \cites{Fischler-Sprang-Zudilin, Lai-Yu} the $0$-th coefficient has a further factors $n+1$ in the denominator. Lai and Sprang \cite{Lai-Sprang} choose $p_n=n+1$ as a test-prime for their irrationality criterion, because it does not appear in the denominators of the coefficients $\rho_i$. Thus, in their situation, it suffices to prove that 
        \[
            v_{n+1}\left( d_{n+1}^{s+1}\rho_{0,\theta} \right)<s+1
        \]
        in order to apply Lemma~\ref{lemma: irrationality criterion}. They then show that the left hand side is bounded above by $|\mathcal{Z}_B|$, which is smaller than $s$. \\
        In our case, however, since $p_n$ may appear in the denominators of both $\rho_{0,\theta}$ and $\rho_{2}$, we actually need
        \[
            v_{p_n}\left( \delta_{n}^{s+1}\rho_{0,\theta} \right)<2.
        \]
        The upper bound $|\mathcal{Z}_B|$ is therefore no longer useful in our context. For this reason, we will compute the $p_n$-adic valuation of $\rho_{0,\theta}$ precisely.
    \end{remark}
     
     \medskip
    
    \subsection{The $p_n$-adic valuation of $a_{i,h}$} $ $\\

    For $i\in \{1,\dots,s\}$ and $h\in \{0,\dots, n\}$ we compute precisely the $p_n$-adic valuation of $a_{i,h}$. From the previous arguments, we know that $v_{p_n}(a_{i,h})\ge -(s+1-i)$. In Proposition~\ref{proposition: p-adic valuation of a_i,h} we give a criterion for which this lower bound turns into an equality.\\
    Let us define, for $y=0,\dots,r_2\mu_n$,
    \[
        H^{(y)}_n=\{ h\in \{0,\dots, n\} \mid h=x+yp_n, \; 0\le x\le n-r_2\mu_np_n \}.
    \]
    By Lemma~\ref{lemma: properties of auxiliary primes}, the binomial coefficient $\binom{n}{h}$ is not a multiple of $p_n$ if and only if $h\in H_n^{(y)}$ for some $y\in\{0,\dots, r_2\mu_n\}$. \\
    Note that for $h\in H^{(y)}_n$ we have 
    \[
         \frac{h}{p_n} = \frac{x}{p_n}+y \to y
    \]
    as $n\to \infty$, because $x\le n-r_2\mu_np_n =o(p_n)$. For the moment, fix $y\in \{0,\dots, \mu_n r_2\}$ and $h\in H^{(y)}_n$. Note that
    \[
         v_{p_n}(\kappa_b^{(n)})=-(2r_1+Tr_2)\mu_n.
    \]
    To split up our computation, we first define
    \[
            K_\ell(h)= 
            \begin{cases}
                \Delta_\ell \left( \prod_{j=0}^{(2r+T)n-1} (b(t-h)+a-brn+bj \right) & \text{if $b\neq 1$}, \\
                \Delta_\ell \left( \prod_{j=0}^{(2r+T)n} (t-h+1-rn+j \right) & \text{if $b= 1$}.
            \end{cases}
    \]
    which depends on $n,\ell,b,a,h$. Notice that the difference between the two cases consists only in the fact that for $b=1$, hence $a=1$, there is one more factor in the product. In particular, we have
    \[
        \Delta_\ell(F_{b,a}(t-h))=\kappa_b^{(n)}K_\ell(h).
    \]
    Thus, $v_{p_n}(\Delta_\ell(F_{b,a}(t-h)))=v_{p_n}(K_\ell(h))+v_{p_n}\left(\kappa_b^{(n)}\right)$. Moreover,
    \[
        K_\ell(h)=b^\ell \sum_{\substack{J\subseteq \{-rn,\dots, (r+T)n-1 \}\\ \#J=\ell} } \prod_{\substack{j=-rn \\ j\notin J}}^{(r+T)n-1}(-bh+a+bj),
    \]
    the product running up to $(r+T)n$ if $b=1$. 
    For $m\in \Z$ the numbers
    \[
        p_n(-a+bm)
    \]
    are all the multiples of $p_n$ congruent to $a$ modulo $b$, because we have $r_2b\mid P_B \mid n$ and $p_n\equiv -1\modulo{P_B}$. We want to understand for which $m\in \Z$ such multiple appears in the above product.
    \begin{lemma}
        \label{lemma: technical lemma for J}
        Let $n\in \mathcal{N}$ sufficiently large. The multiples of $p_n$ congruent to $a$ modulo $b$ that lie in the interval
        \[
            [-bh+a-brn,  -bh+a+b(r+T)n-b]
        \]
        are precisely the numbers of the form $p_n(-a+bm)$ for all $m\in \Z$ in the interval
        \[
            \left[\frac{a}{b}-r_1\mu_n-y, \frac{a}{b}+(r_1+Tr_2)\mu_n-y\right].
        \]
    \end{lemma}
    \begin{proof}
        Let us prove the upper bound for $m$ first. Set 
        \[
            A=-bh+a+b(r+T)n-b
        \]
        Suppose that $m>\frac{a}{b}+(r_1+Tr_2)\mu_n-y$. Since $m\in \Z$, we have $m\ge \frac{a}{b}+(r_1+Tr_2)\mu_n-y+1$. As a result,
        \[
            p_n(-a+bm)\ge b((r_1+Tr_2)\mu_n+1)p_n-byp_n=B.
        \]
        Suppose that $p_n(-a+bm)\le A$. Then, $ B \le p_n(-a+bm)\le A$. However,
        \begin{align*}
            B-A \sim  p_n+b-a+bx.
        \end{align*}
        The latter quantity is positive and unbounded, so $B>A$ for all $n\in \mathcal{N}$ sufficiently large, a contradiction. It follows that 
        \[
            m\le \frac{a}{b}+(r_1+Tr_2)\mu_n-y.
        \]
        We now wish to show that, when equality holds, we have indeed $p_n(-a+bm)\le A$. In this case, note that
        \begin{align*}
            p_n(-a+bm) & \le b(r+T)n -b(r_1+Tr_2)\mu_ng(n)-byp_n.
        \end{align*}
        By \eqref{eq: estimate for difference n and prime}, we have
        \begin{align*}
            A- p_n(-a+bm) & \ge  a-b-bx+b(r_1+Tr_2)\mu_ng(n) \\
            & \ge a-b +br_2\left(r+T-1-\frac{1}{\delta}\right)\mu_ng(n).
        \end{align*}
        Our choice of $\delta$ in Section 1 ensures that the latter quantity is positive. The lower bound for $m$ is completely analogous.
    \end{proof}
    Let us set
    \[
        J_{b,a}^{(y)}=\Z\cap \left[ \frac{a}{b}-r_1\mu_n-y,\frac{a}{b}+(r_1+Tr_2)\mu_n -y\right] \setminus \left\{\frac{a}{b} \right\},
    \]
    which encloses all values of $m$ for which $p_n(-a+bm)$ is a non-zero multiple of $p_n$ under the conditions of the previous lemma. We observe that $J_{b,a}^{(y)}$ depends only on $y$ and on whether $\frac{a}{b}$ is integer or not, but its cardinality is constant in all cases and equals 
    \[
        \#J_{b,a}^{(y)}=(2r_1+Tr_2)\mu_n.
    \]
    \begin{lemma}
        \label{lemma: valuation of K_l(h)}
        Let $h\in H^{(y)}_n$. Suppose that the integer
        \[
            g_{\ell,b,a}(y)=\sum_{\substack{J\subseteq J_{b,a}^{(y)} \\ \#J=\ell }} \prod_{j\in J} (-a+bj)
        \]
        is not zero. Then, for $0\le \ell \le  (2r_1+Tr_2)\mu_n $, we have $v_{p_n}(d_n^\ell K_\ell(h))= (2r_1+Tr_2)\mu_n $. Moreover, for $\ell\ge  (2r_1+Tr_2)\mu_n +1$, we have $v_{p_n}(d_n^\ell K_\ell(h))\ge  (2r_1+Tr_2)\mu_n +1$. 
    \end{lemma}
    \begin{proof}
        Since $K_\ell(h)$ is an integer and $v_{p_n}(d_n^\ell)=\ell$, we have $v_{p_n}(d_n^\ell K_\ell(h))\ge \ell$, which trivially proves the claim for $\ell\ge  (2r_1+Tr_2)\mu_n +1$. \\
        Assume now that $0\le \ell \le  (2r_1+Tr_2)\mu_n $. Let $M_{b,a}$ be the set of multiples of $p_n$ between $-bh+a-brn$ and $-bh+a+b(r+T)n-b$ that are congruent to $a$ modulo $b$. By our previous observation, we have $\#M_{b,a}= (2r_1+Tr_2)\mu_n$. Let us set for short $\alpha=a-hb$ and $\beta=b$. \\
        We may then write 
        \[
            b^{-\ell}K_\ell(h)= \sum_{\substack{J_1\subseteq \{-rn,\dots, (r+T)n-1 \}\\ \#J_1=\ell, \, J_1\subseteq M_{b,a}} } \prod_{\substack{j=-rn \\ j\notin J_1}}^{(r+T)n-1}(\alpha+\beta j)+ \sum_{\substack{J_2\subseteq \{-rn,\dots, (r+T)n-1 \}\\ \#J_2=\ell, \, J_2\not\subseteq M_{b,a}} } \prod_{\substack{j=-rn \\ j\notin J_2}}^{(r+T)n-1}(\alpha+\beta j).
        \]
        Let us look at the sum over $J_2$, so $J_2\not\subseteq M_{b,a}$.  Notice that the assumption $0\le \ell \le  (2r_1+Tr_2)\mu_n $ ensures that $J_2$ cannot properly contain $M_{b,a}$ for cardinality reasons. As a result, 
        \[
            \# J_2\cap M_{b,a}\le \ell -1,
        \]
        hence the $p_n$-adic valuation of the internal product is at least $ (2r_1+Tr_2)\mu_n -\ell +1$. Thus, $b^{-\ell}K_\ell(h)$ is congruent to the sum over $J_1$ modulo $p_n^{ (2r_1+Tr_2)\mu_n -\ell+1}$, which equals
        \begin{align*}
            \prod_{ \substack{j=-rn \\ p_n \nmid \alpha+\beta j} }^{(r+T)n-1} (\alpha+\beta j)\sum_{\substack{M\subseteq M_{b,a} \\ \#M=\ell }} \prod_{m\in M} m.
        \end{align*}
        Recall that the numbers in $M_{b,a}$ are of the form $p_n(-a+bj)$ with $j\in J_{b,a}^{(y)}$. Thus,
        \[
            b^{-\ell}K_\ell(h)\equiv 
            \left(\prod_{ \substack{j=-rn \\ p_n \nmid \alpha+\beta j} }^{(r+T)n-1} (\alpha+\beta j)\right) 
            p_n^{ (2r_1+Tr_2)\mu_n -\ell}
            \sum_{\substack{J\subseteq J_{b,a}^{(y)} \\ \#J=\ell }} \prod_{j\in J} (-a+bj) \modulo{p_n^{ (2r_1+Tr_2)\mu_n-\ell+1}}.
        \]
        The first product is not divisible by $p_n$. Moreover, the final sum equals $g_{\ell,b,a}(y)$, which is not zero by assumption. \\
        It is immediate to see that there are constants $c_1,c_2>0$, independent of $n$, such that
        \[
            \left| g_{\ell,b,a}(y) \right| \le (c_1\mu_n)^{c_2\mu_n}.
        \]
        Given our choice $\mu_n=\log(\log(n))$, it follows that $ |g_{\ell,b,a}(y)| =o(p_n)$, so for all sufficiently large $n$ we have $|g_{\ell,b,a}(y)|<p_n$. In particular, $g_{\ell,b,a}(y)$ can be divisible by $p_n$ only if it equals $0$. This does not occur thanks to the assumption. This proves that the $p_n$-adic valuation of $K_\ell(h)$ equals $ (2r_1+Tr_2)\mu_n -\ell$, whence the first claim follows.
    \end{proof}
    The previous lemma computes the $p_n$-adic valuation of $K_\ell(h)$ for certain $\ell$, but only under the assumption that $h\in H_n^{(y)}$. In the sequel, we will also need a lower bound on this valuation for general $\ell$, which we now present. 
    \begin{lemma}
        For all $h\in \{0,\dots, n\}$ we have
        \[
            v_{p_n}(K_\ell(h))\ge  (2r_1+Tr_2)\mu_n -\ell.
        \]
    \end{lemma}
    \begin{proof}
        Even when $h\in \{0,\dots,n\}$ is not constrained by $h\in H^{(y)}_{n}$, we can still write similar expressions for $K_\ell(h)$ as above. The only difference is that $J_{b,a}^{ (h)}$ has to be replaced by the set $J$ of $m\in \Z$ such that
        \[
            -bh+a-brn \le p_n(-a+bm)\le -bh+a+b(r+T)n-b.
        \]
        The same argument as in Lemma~\ref{lemma: technical lemma for J} then shows that 
        \[
            m\in J=\Z\cap \left[ \frac{a}{b}-r_1\mu_n-\frac{h}{p_n}, \frac{a}{b}+(r_1+Tr_2)\mu_n-\frac{h}{p_n}\right].
        \]
        Suppose that $b\neq 1$. We observe that 
        \[
            \frac{a}{b}-\frac{h}{p_n}\not\in \Z.
        \]
        Indeed, if this were not the case, we would have
        \[
            bp_n\mid ap_n-bh,
        \]
        from which $b\mid ap_n-bh $. However, since $p_n\equiv -1\modulo{P_B}$, this divisibility is impossible. In particular, $\#J=(2r_1+Tr_2)\mu_n$. 
        
        \smallskip
        
        Suppose now that $b=a=1$. We may assume $\frac{h}{p_n}$ is not an integer because this case is already included in $h\in H^{(y)}_n$. Then 
        \[
            \#J= \Z\cap \left[ -r_1\mu_n-\frac{h}{p_n}, (r_1+Tr_2)\mu_n-\frac{h}{p_n}\right]=(2r_1+Tr_2)\mu_n.
        \]
        The statement now follows exactly as in the Lemma~\ref{lemma: valuation of K_l(h)}.
    \end{proof}
    \begin{remark}
        This argument shows that $\#J$ and $\#J^{(y)}_{b,a}$ are equal to $(2r_1+Tr_2)\mu_n$. Our more precise argument for the case $h\in H^{(y)}_n$ actually shows that the set $J^{(y)}_{b,a}$ depends only on $y$. This will be crucial later in our exposition.
    \end{remark}
    We return to $h\in H^{(y)}_n$.  The proof of Lemma~\ref{lemma: valuation of K_l(h)} actually gives us more information on $K_\ell(h)$. Let us define
    \[
        f_{b,a}^{(n)}(h)=
        \begin{dcases}
            \prod_{ \substack{j=-rn \\ p_n \nmid a-bh+b j} }^{(r+T)n-1} (a-bh+b j) & \text{if $b\neq 1 $}; \\
            \prod_{ \substack{j=-rn \\ p_n \nmid 1-h+j} }^{(r+T)n} (1-h+j) & \text{if $b\neq 1 $}. \\
        \end{dcases}
    \]
    We then have
    \[  
        K_\ell(h)\equiv p_n^{v_{b,a}-\ell}b^\ell f_{b,a}^{(n)}(h)g_{\ell,b,a}(y) \modulo{p_n^{v_{b,a}-\ell+1}},
    \]
    where $f_{b,a}^{(n)}(h)$ does not depend on $\ell$ and is not divisible by $p_n$, while $g_{\ell,b,a}(y)$ does not depend on $n$. For later reference we define 
    \[
        F^{(n)}(h)=\prod_{\frac{a}{b}\in \mathcal{Z}_B} f_{a,b}^{(n)}(h), \qquad G_{\underline{\ell}}(y)= \prod_{\frac{a}{b}\in \mathcal{Z}_B} g_{\ell_{b,a},b,a}(y),
    \]
    where the $\underline{\ell}=(\ell_{b,a})_{b,a}$ is a sequence of integers that are parametrized by the elements of $\mathcal{Z}_B$ and satisfy $0\le\ell_{b,a}\le  (2r_1+Tr_2)\mu_n $. \\

    Let us turn to the $p_n$-adic valuation of $a_{i,h}$. By~\eqref{eq: explicit expression for a_i,h}, we have
    \begin{align}
    \label{eq: first lower bound for the valuation of a_i,h}
        v_{p_n}(a_{i,h})  \ge & (s+1)v_{p_n}\left(\binom{n}{h} \right) + \\
                            & + \min_{\underline{\ell}} \left\{ \sum_{\frac{a}{b}\in \mathcal{Z}_B} v_{p_n} \left(\Delta_{\ell_{b,a}}\left( \widetilde{F}_{b,a}(t-h)\right) \right)+ \sum_{\substack{k=0 \\ k\neq h}}^n v_{p_n}\left(\frac{(-1)^{\ell_k}\binom{s+\ell_k}{\ell_k}}{T^{\ell_k}(k-h)^{\ell_k}}\right) \right\}. \nonumber
    \end{align}
    Given that, by Lemma~\ref{lemma: valuation of K_l(h)},  
    \[
        v_{p_n}\left(\Delta_{\ell_{b,a}}\left( \widetilde{F}_{b,a}(t-h)\right) \right) = 
        v_{p_n}(K_\ell(h))-(2r_1+Tr_2)\mu_n\ge -\ell,
    \]
    the $p_n$-adic valuation of the minimum above is at least
    \[
        -\sum_{\frac{a}{b}\in \mathcal{Z}_B}\ell_{a,b}-\sum_{\substack{k=0 \\ k\neq h}}^n \ell_k = -(s+1-i).
    \]
    Set 
    \[
        H_n=\bigcup_{y=0}^{r_2\mu_n}H^{(y)}_n.    
    \]
    Suppose that $h\notin H_n$. Then, by Lemma~\ref{lemma: properties of auxiliary primes}, the binomial coefficient $\binom{n}{h}$ has $p_n$-adic valuation at least $1$. This gives a further summand $s+1$ in \eqref{eq: first lower bound for the valuation of a_i,h}. Hence, when $h\not\in H_n$, 
    \[
        v_{p_n}(a_{i,h})\ge s+1-(s+1-i)=i>-(s+1-i).
    \]
    This proves the following lemma:
    \begin{lemma}
        If $h\notin H_n$, then $v_{p_n}(a_{i,h})>-(s+1-i)$. 
    \end{lemma}
    Let us therefore assume that $h\in H_n$, say $h\in H_n^{(y)}$ for $y\in \{0,\dots, r_2\mu_n\}$. Set for short
    \[
        \kappa^{(n)}=\prod_{\frac{a}{b}\in\mathcal{Z}_B} \kappa_b^{(n)},
    \]
    We wish to isolate the terms in \eqref{eq: first lower bound for the valuation of a_i,h} that have minimal $p_n$-adic valuation. The set of $k\in \{0,\dots, n\}$ such that $k\neq h$ and $p_n\mid k-h$ is given by 
    \[
        \{ x+\nu p_n \mid \nu=0,\dots,r_2\mu_n, \, \nu\neq y\}.
    \]
    Consider the set $\Lambda_i$ of tuples
    \[
        \underline{\lambda}=(\underline{\ell}',\underline{\ell}), \qquad \underline{\ell}'=(\ell'_{\nu})_{\nu}, \qquad \underline{\ell}=(\ell_{b,a})_{\frac{a}{b}\in \mathcal{Z}_B},
    \]
    with $\nu\in\{0,\dots, r_2\mu_n\}\setminus \{y\}$, such that 
    \[
        \sum_{\substack{\nu=0 \\ \nu \neq y}}^{r_2\mu_n} \ell'_\nu + \sum_{b,a}\ell_{b,a} = s+1-i.
    \]
    In view of \eqref{eq: first lower bound for the valuation of a_i,h} and Lemma~\ref{lemma: valuation of K_l(h)}, we deduce that
    \begin{equation}
        \label{eq. 1}
        a_{i,h} \equiv p_n^{-(s+1-i)} \kappa^{(n)}F^{(n)}(h)\binom{n}{h}^{s+1}\sum_{\underline{\lambda}\in \Lambda_i} G_{\underline{\ell}}(y)
        \sum_{\substack{\nu=0 \\ \nu \neq y}}^{r_2\mu_n}(-1)^{\ell_{\nu}'}\binom{s+\ell_{\nu}'}{s}\frac{1}{T^{\ell'_{\nu}}(\nu-y)^{\ell'_\nu}},
    \end{equation}
    the congruence taking place modulo $p_n^{-(s+1-i)+1}$.
    \begin{proposition}
        \label{proposition: p-adic valuation of a_i,h}
        For $h\in H_n^{(y)}$, if the quantity
        \[
            Q_i(y)=\sum_{\underline{\lambda}\in \Lambda_i} G_{\underline{\ell}}(y)
        \sum_{\substack{\nu=0 \\ \nu \neq y}}^{r_2\mu_n}(-1)^{\ell_{\nu}'}\binom{s+\ell_{\nu}'}{s}\frac{1}{T^{\ell'_{\nu}}(\nu-y)^{\ell'_\nu}},
        \]
        is not zero, then $v_{p_n}(a_{i,h})=-(s+1-i)$.
    \end{proposition}
    \begin{proof}
        Since $0<T^{s+1-i}<p_n$, the $p_n$-adic valuations of $a_{i,h}$ and $T^{s+1-i}a_{i,h}$ are equal. The right hand side of (\ref{eq. 1}) becomes an integer when multiplied by 
        \[
            D_n=\left(Tp_n\text{lcm}(1,\dots,r_2\mu_n)\right)^{s+1-i}.
        \]
        Since $h\in H_n^{(y)}$, the binomial coefficient appearing in (\ref{eq. 1}) is not divisible by $p_n$. Also the integer $F^{(n)}(h)$ is not a multiple of $p_n$. The $p_n$-adic valuation of $\kappa^{(n)}$ is zero. \\
        It remains to check that $D_nQ_i(y)$ has trivial $p_n$-adic valuation. Note that $p_n^{-(s+1-i)}D_n$ has exponential growth in $\mu_n$. It is straightforward to see that there are constants $c_1,c_2>0$ independent of $n$ such that 
        \[
            \left| T^{s+1-i}Q_i(y) \right|\le (c_1\mu_n)^{c_2\mu_n}=o(p_n)
        \]
        It follows that $|T^{s+1-i}Q_i(y)|<p_n$, so its $p_n$-adic valuation vanishes when $Q_i(y)\neq 0$.
    \end{proof}

    \medskip

    \subsection{The $p_n$-adic valuation of $\rho_i$} $ $\\

   We turn to the study of $\rho_i$. In Proposition~\ref{proposition: p-adic valuation of rho_i} we give a criterion for which the $p_n$-adic valuation of $\rho_i$ is as small as possible, that is, $-(s+1-i)$. This result is not strictly necessary for the proof of Theorem~\ref{theorem: main}. However, it is very straightforward to establish on the way of studying the $p_n$-adic valuation of $\rho_{0,\theta}$, so we include it nonetheless. \\
   If $i\equiv \sigma(\chi)\modulo{2}$, then $\rho_i=0$, so we assume $i\equiv \sigma(\chi)+1\modulo{2}$. Set 
    \[
        \widetilde{H}_n= \bigcup_{y=0}^{\left\lfloor \frac{r_2\mu_n}{2}\right\rfloor} H^{(y)}_n,
    \]
    Modulo $p_n^{-(s+1-i)}$, we see that $\rho_i$ is congruent to the sum of the $a_{i,h}$ for $h\in \widetilde{H}_n$.
    \[
        \rho_i=\sum_{h=0}^n a_{i,h}\equiv \sum_{h\in H_n} a_ {i,h}.
    \]
    In order to prove that the valuation of $\rho_i$ equals $-(s+1-i)$, we therefore need to show that the number
    \begin{equation}
        \label{eq: technical quantity for rho_i}
        \sum_{y=0}^{\left\lfloor \frac{r_2\mu_n}{2}\right\rfloor} 
        Q_i(y)
        \sum_{h\in H^{(y)}_n} \binom{n}{h}^{s+1}F^{(n)}(h)
    \end{equation}
    has trivial $p_n$-adic valuation. For $h\in H^{(y)}_n$ we have $h\equiv x \modulo{p_n}$ with $0\le x \le n-r_2 \mu_n p_n$ and it follows easily that 
    \[
        F^{(n)}(x+yp_n)\equiv F^{(n)}(x) \modulo{p_n}.
    \]
    On the other hand, by Lucas's theorem,
    \[
        \binom{n}{x+yp_n}\equiv \binom{n-r_2p_n}{x}\binom{r_2}{y}\modulo{p_n}.
    \]
    As a result, modulo $p_n$, we deduce that \eqref{eq: technical quantity for rho_i} is congruent to
    \[
        \left(\sum_{y=0}^{\left\lfloor \frac{r_2}{2}\right\rfloor} 
        Q_i(y)\binom{r_2}{y}^{s+1}\right) 
        \left(\sum_{x=0}^{n-r_2\mu_np_n}\binom{n-r_2p_n}{x}^{s+1}F^{(n)}(x) \right).
    \]
    To simplify this quantity further, we define
    \[
        \widetilde{f}_{b,a}^{(n)} = \prod_{\substack{ j=0 \\ p_n \nmid a-brn+b j}}^{(2r+T)n-1-p_n } (a-brn+bj)
    \]
    which does not depend on $x$. Different choices of $x$ provide a translation in the factors of $f_{b,a}^{(n)}(x)$. Thus, for $b\neq 1$, we have
    \begin{align*}
        f_{b,a}^{(n)}(x)&=
        \prod_{ \substack{j=-x \\ p_n \nmid a-brn+b j} }^{(2r+T)n-1-x} (a-brn+b j) \\
        &= 
        \widetilde{f}_{b,a}^{(n)} 
        \prod_{ \substack{ j=1  \\ p_n \nmid a-brn+b j}}^{x} (a-brn-bj)
        \prod_{\substack{ j=(2r+T)n-(n-r_2\mu_np_n) \\ p_n \nmid a-brn+b j}}^{(2r+T)n-1-x } (a-brn+bj).
    \end{align*}
    If $b=a=1$, then this equality holds with an additional factor $1+(r+T)n-x$. For $b\neq 1$,
    \begin{align*}
        f_{b,a}^{(n)}(x)=\widetilde{f}_{b,a}^{(n)}\prod_{ j=1  }^{x} (a-brn-bj) \prod_{j=1}^{n-r_2\mu_np_n-x}(a-b+b(2r+T)n-br_2\mu_np_n+bj)
    \end{align*}
    and for $b=1$ a factor $1+(r+T)n-x$ also appears. 
    Let us set
    \begin{align*}
        \xi_n  =& \sum_{x=0}^{n-r_2\mu_n p_n}\binom{n-r_2\mu_n p_n}{x}^{s+1}\prod_{\frac{a}{b}\in \mathcal{Z}_B} \prod_{ j=1  }^{x} \left(a-br_1\left( \frac{n}{r_2}-\mu_np_n\right)-bj\right) \times\\
        & \times \prod_{j=1}^{n-r_2\mu_np_n-x}\left(a-b+b(2r_1+Tr_2)\left( \frac{n}{r_2}-\mu_np_n\right)+bj\right).
    \end{align*}
    Our discussion shows that
    \begin{align*}
        p_n^{s+1-i}\rho_i \;\equiv & \;2\kappa^{(n)}
        \xi_n  \, 
        \prod_{\frac{a}{b}\in \mathcal{Z_B}}      \widetilde{f}_{b,a}^{(n)}
        \sum_{y=0}^{\left\lfloor \frac{r_2}{2}\right\rfloor} 
        Q_i(y)\binom{r_2}{y}^{s+1} \; \modulo{p_n}.
    \end{align*}
    \begin{lemma}
        \label{lemma: digamma}
        The number $\xi_n $ is not divisible by $p_n$. 
    \end{lemma}
    \begin{proof}
        We start by observing that $\xi_n > 0$. Indeed, 
        \[
            (-1)^x\prod_{ j=1  }^{x} \left(a-br_1\left( \frac{n}{r_2}-\mu_np_n\right)-bj\right)>0,
        \]
        since all the factors in the product are negative. The remaining factors in the definition of $\xi_n $ are positive, so each summand in $\xi_n $ has sign
        \[
            (-1)^{|\mathcal{Z}_B|x}.
        \]
        By the very definition of $\mathcal{Z}_B$, we have
        \[
            |\mathcal{Z}_B| = \sum_{\phi(b)\le B} \phi(b).
        \]
        The numbers $b=1$ and $b=2$ are the only numbers with $\phi(b)$ odd. Thus, $|\mathcal{Z}_B|$ must be even, so $\xi_n > 0$. On the other hand, there are constants $C_1,C_2>0$ independent of $n$ such that
        \[
            \xi_n  < C_1(\mu_nf(n))^{C_2f(n)} .
        \]
        With the choices $\mu_n=\lfloor \log\log(n)\rfloor$ and $f(n)=\log\log\log(n)$, we have $\xi_n =o(p_n)$. In particular, $0<\xi_n<p_n$ for $n$ sufficiently large, which proves the claim.
    \end{proof}
    Returning to $\rho_i$, the product 
    \[
        \prod_{\frac{a}{b}\in \mathcal{Z_B}}      \widetilde{f}_{b,a}^{(n)}
    \]
    is not divisible by $p_n$, while the remaining quantity
    \[
        \sum_{y=0}^{\left\lfloor \frac{r_2\mu_n}{2}\right\rfloor} 
        Q_i(y)\binom{r_2}{y}^{s+1}
    \]
    is too small to be divisible by $p_n$, after making it an integer by multiplying it by $T^{s+1-i}$. We have thus proved the following result:
    \begin{proposition}
        \label{proposition: p-adic valuation of rho_i}
        If $i\equiv \sigma(\chi)+1\modulo{2}$ and
        \[
            Q_i'=
            \sum_{y=0}^{\left\lfloor \frac{r_2\mu_n}{2}\right\rfloor} 
        Q_i(y)\binom{r_2}{y}^{s+1}
        \]
        is not zero, then $v_{p_n}(\rho_i)=-(s+1-i)$.
    \end{proposition}

    \medskip
    
    \subsection{Lower bound for the denominator of $\rho_{0,\theta}$ and $\rho_{0,b}'$} $ $ \\
    
    Take $\theta=\frac{a}{b}\in \mathcal{Z}_B$ with $\gcd(a,b)=1$.
    \[
        \rho_{0,\theta}=-\sum_{k=0}^n\sum_{\ell=0}^{Tk}\sum_{i=1}^s\frac{a_{i,k} \, \chi_\theta(\ell)}{(\ell+\theta)^i}.
    \]
    In Proposition~\ref{proposition: p-adic valuation of rho_0,theta} we give a criterion for the $p_n$-adic valuation of $\rho_{0,\theta}$ to be as small as possible, that is, $-s-1$. This is crucial for the proof of Theorem~\ref{theorem: main}, especially for the application of the irrationality criterion of \cite{Lai-Sprang}. \\
    In order for the $p_n$-adic valuation of $\rho_{0,\theta}$, we focus on the summands with smallest valuation. This requires that $p_n$ divides $b\ell+a$ and $k\in H_n$, otherwise the $p_n$-adic valuation of $a_{i,k}$ is greater than $-(s+1-i)$. Assume first that $b\neq 1$. The condition $k\in H_n$ means $k=x+yp_n$ with $y\in {0,\dots, r_2\mu_n}$ and $ 0\le x\le n-r_2\mu_np_n$. Then, for $p_n$ to divide $b\ell +a$ we need $a+b\ell=p_n(-a+bj)$ and $0\le \ell \le Tk=Typ_n+Tx$. As a result,
    \[
        \frac{a}{b}+\frac{a}{bp_n} \le j \le \frac{a}{b}+\frac{a}{bp_n}+Ty+\frac{Tx}{p_n}.
    \]
    Since $y\in \Z$ and $x=o(p_n)$, we deduce that $1\le j\le Ty$. This shows that, modulo $p_n^{-(s+1)+1}$, 
    \[
        \rho_{0,\theta} \equiv
            \sum_{i=1}^s\sum_{y=0}^{r_2\mu_n} \sum_{x=0}^{n-r_2\mu_np_n}\sum_{j=1}^{Ty} 
            \chi_\theta\left(-\frac{a}{b}+p_n\left(-\frac{a}{b}+j \right)\right)
            \frac{b^ia_{i,x+yp_n}}{(-a+bj)^ip_n^i}.
    \]
    Since $p_n\equiv -1\modulo{T}$ and $\chi_\theta$ has period $T$, it follows that 
    \[
        \chi_\theta\left(-\frac{a}{b}+p_n\left(-\frac{a}{b}+j \right)\right)=\chi_\theta(-j).
    \]
    Thus, as in the case of $\rho_i$, we have
    \[
        p_n^{s+1}\rho_{0,\theta} \equiv  \kappa^{(n)} \xi_n
        \prod_{\frac{\alpha}{\beta}\in \mathcal{Z}_B} \widetilde{f}_{\beta,\alpha}^{(n)} 
        \sum_{i=1}^s 
        \sum_{y=0}^{r_2\mu_n} Q_i(y) \binom{r_2}{y}^{s+1}
        \sum_{j=1}^{Ty} 
        \frac{b^i\chi_\theta(-j)}{(-a+bj)^i} 
        \qquad  \modulo{ p_n}.
    \]
    Previously, we defined $Q_i(y)$ only for $i\equiv \sigma(\chi)+1\modulo{2}$, but an analogous definition can be extended to the remaining case. \\
    Suppose instead that $b=a=1$. We obtain a completely analogous identity modulo $p_n$, but the sum over $j$ runs from $2$ to $Ty+1$. An analogous argument to the one in Proposition~\ref{proposition: p-adic valuation of a_i,h} proves the following result:
    \begin{proposition}
        \label{proposition: p-adic valuation of rho_0,theta}
        For $\theta\in \mathcal{Z}_B$, $\theta=\frac{a}{b}$ with $\gcd(a,b)=1$, define 
        \[
            Q_\theta'=
            \begin{dcases}
                \sum_{i=1}^s 
        \sum_{y=0}^{r_2\mu_n} Q_i(y) \binom{r_2}{y}^{s+1}
        \sum_{j=1}^{Ty} 
        \frac{b^i\chi_\theta(-j)}{(-a+bj)^i} 
         & \text{if $b\neq 1$}; \\
                \sum_{i=1}^s 
        \sum_{y=0}^{r_2\mu_n} Q_i(y) \binom{r_2}{y}^{s+1}
        \sum_{j=2}^{Ty+1} 
        \frac{b^i\chi_\theta(-j)}{(-a+bj)^i} 
         & \text{if $b=a= 1$}.
            \end{dcases}
        \]
        If $Q_\theta'$ is not zero, then $v_{p_n}(\rho_{0,\theta})=-(s+1)$.
    \end{proposition}
    Recall now that for all $b\in \Psi_B$
    \[
        \rho_{0,b}'= \sum_{h=1}^b \chi(\gcd(h,b)) \rho_{0,{\frac{h}{b}}}.
    \]
    \begin{corollary}
        If 
        \[
            Q_b''=\sum_{h=1}^b \chi(\gcd(h,b)) Q'_{\frac{h}{b}}
        \]
        is not zero, then $v_{p_n}(\rho_{0,b}')$ equals $-(s+1)$.
    \end{corollary}
    Observe that $\chi_{b,a}(j)=\chi(\gcd(b,a))\chi_{\frac{a}{b}}(j)$ and
    \[
        \sum_{a=1}^ {b-1}\sum_{j=1}^{Ty} 
        \frac{b^i\chi_{b,a}(-j)}{(-a+bj)^i}+
        \sum_{j=2}^{Ty+1} 
        \frac{b^i\chi_{b,b}(-j)}{(-1+j)^i} = b^i\sum_{j=1}^{Tby} \frac{\chi(-j)}{j^i}.
    \]
    As a result,
    \[
        Q_b''=\sum_{i=1}^s 
        \sum_{y=0}^{r_2\mu_n} Q_i(y) \binom{r_2}{y}^{s+1}b^i\sum_{j=1}^{Tby} \frac{\chi(-j)}{j^i}.
    \]
    Finally, we turn to 
    \[
        \rho_0''=\sum_{b\in \Psi_B} w_b \rho_{0,b}', 
    \]
    where $w_b$ are suitable integers to be defined.
    \begin{corollary}
        \label{corollary: p-adic valuation of Q tilde}
        If 
        \[
            \widetilde{Q} = \sum_{b\in \Psi_B} w_b Q_b''
        \]
        is not zero, then $v_{p_n}(\rho_0'')=-(s+1)$.
    \end{corollary}

    \medskip

    \subsection{Non-vanishing} $ $\\

    The goal of this section is to prove that the assumption $\widetilde{Q}\neq 0$ is satisfied for a suitable choice of $w_b$. More precisely, let $\psi=|\mathcal{Z_B}|-1$ and write $b_0,\dots, b_\psi$ for the elements of $\Psi_B$ in any order. Choose distinct integers $i_0,\dots, i_\psi \in \{2,\dots, s\}$. Consider that the following $(\psi+1)\times(\psi+1)$ matrix:
    \[
        M_n=
        \begin{pmatrix}
            Q_{b_0}'' & b_0^{i_1} & \dots & b_0^{i_\psi} \\
            Q_{b_1}'' & b_1^{i_1} & \dots & b_1^{i_\psi} \\
            \vdots & \vdots & \ddots & \vdots \\
            Q_{b_\psi}'' & b_\psi^{i_1} & \dots & b_\psi^{i_\psi} \\
        \end{pmatrix}.
    \]
    We will prove the following result:
    \begin{proposition}
        \label{proposition: det ist not zero}
        For all sufficiently large $n\in\mathcal{N}$ the matrix $M_n$ is non-singular.
    \end{proposition}
    In particular, this proposition ensures that there are integers $(w_b)_{b\in \Psi_B}$ such that for all $j=1,\dots,\psi$
    \[
        \sum_{b\in \Psi_B} w_bQ_b'''\neq 0, \qquad \sum_{b\in \Psi_B} w_b b^{i_j}=0
    \]
    for all sufficiently large $n\in\mathcal{N}$. The first condition tells us that $\widetilde{Q}\neq 0$ for this choice of $w_b$'s, so the conclusion of Corollary~\ref{corollary: p-adic valuation of Q tilde} holds. The invertibility of the matrix $M_n$ will be used at the end of the proof of Theorem~\ref{theorem: main} to eliminate many $L$-values from $r_n''$. \\ 
    The main analytic ingredient for proving Proposition~\ref{proposition: det ist not zero} is the following:
    \begin{lemma}
        \label{lemma: asymptotics for Q_b''}
        For all sufficiently large $n\in\mathcal{N}$ and all $b\in \Psi_B$ we have $Q_1''\neq 0$ and 
        \[
            \lim_{n\to \infty} \frac{Q_b''}{Q_1''}=b.
        \]     
    \end{lemma}
    We first observe that Proposition~\ref{proposition: det ist not zero} follows immediately from this lemma.
    \begin{proof}[Proof of Proposition~\ref{proposition: det ist not zero}]
    Observe that, by Lemma\ref{lemma: asymptotics for Q_b''},
    \[
        \det M_n = Q_1'' \det 
        \begin{pmatrix}
            \frac{Q_{b_0}''}{Q_1''} & b_0^{i_1} & \dots & b_0^{i_\psi} \\
            \frac{Q_{b_1}''}{Q_1''} & b_1^{i_1} & \dots & b_1^{i_\psi} \\
            \vdots & \vdots & \ddots & \vdots \\
            \frac{Q_{b_\psi}''}{Q_1''} & b_\psi^{i_1} & \dots & b_\psi^{i_\psi} \\
        \end{pmatrix} =
        Q_1'' \det 
        \begin{pmatrix}
            b_0+o(1) & b_0^{i_1} & \dots & b_0^{i_\psi} \\
            b_1+o(1) & b_1^{i_1} & \dots & b_1^{i_\psi} \\
            \vdots & \vdots & \ddots & \vdots \\
            b_{\psi}+o(1) & b_\psi^{i_1} & \dots & b_\psi^{i_\psi} \\
        \end{pmatrix}.    
    \]
    By \cite[Lemma 4]{Fischler-Sprang-Zudilin}, the matrix 
    \[
        \begin{pmatrix}
            b_0 & b_0^{i_1} & \dots & b_0^{i_\psi} \\
            b_1 & b_1^{i_1} & \dots & b_1^{i_\psi} \\
            \vdots & \vdots & \ddots & \vdots \\
            b_{\psi} & b_\psi^{i_1} & \dots & b_\psi^{i_\psi} \\
        \end{pmatrix}
    \]
    is invertible, because $2\le i_j \le s$. Moreover, the determinant of 
    \[
        \begin{pmatrix}
            o(1) & b_0^{i_1} & \dots & b_0^{i_\psi} \\
            o(1) & b_1^{i_1} & \dots & b_1^{i_\psi} \\
            \vdots & \vdots & \ddots & \vdots \\
            o(1) & b_\psi^{i_1} & \dots & b_\psi^{i_\psi} \\
        \end{pmatrix}
    \]
    goes to zero as $n\to \infty$, because the $b_j$'s and $i_j$'s do not depend on $n$. The claim follows.
    \end{proof}
    Before proving Lemma~\ref{lemma: asymptotics for Q_b''}, we briefly recall the definition of the main quantities for convenience. We have
    \[
        Q_b''=\sum_{i=1}^s 
        \sum_{y=0}^{r_2\mu_n} Q_i(y) \binom{r_2\mu_n}{y}^{s+1}b^i\sum_{j=1}^{Tby} \frac{\chi(-j)}{j^i},
    \]
    where 
    \begin{equation}
        \label{eq: Q_i(y)}
        Q_i(y)=\sum_{\underline{\lambda}\in \Lambda_i} G_{\underline{\ell}}(y)
        \sum_{\substack{\nu=0 \\ \nu \neq y}}^{r_2\mu_n}(-1)^{\ell_{\nu}'}\binom{s+\ell_{\nu}'}{s}\frac{1}{T^{\ell'_{\nu}}(\nu-y)^{\ell'_\nu}}.
    \end{equation}
    Here, the set $\Lambda_i$ consists of the tuples
    \begin{equation}
        \label{eq: Lambda_i}
        \underline{\lambda}=(\underline{\ell}',\underline{\ell}), \qquad \underline{\ell}'=(\ell'_{\nu})_{\nu}, \qquad \underline{\ell}=(\ell_{b,a})_{\frac{a}{b}\in \mathcal{Z}_B},
    \end{equation}
    with $\nu\in\{0,\dots, r_2\mu_n\}\setminus \{y\}$, such that 
    \[
        \sum_{\substack{\nu=0 \\ \nu \neq y}}^{r_2\mu_n} \ell'_\nu + \sum_{b,a}\ell_{b,a} = s+1-i.
    \]
    Moreover, recall that
    \[
        G_{\underline{\ell}}(y)=\prod_{\frac{a}{b}\in \mathcal{Z}_B} g_{\ell_{b,a},b,a}(y), \qquad g_{\ell,b,a}(y)= \sum_{\substack{J\subseteq J_{b,a}^{(y)} \\ \#J=\ell }} \prod_{j\in J} (-a+bj)
    \]
    with 
    \[
        J_{b,a}^{(y)}=
        \begin{cases}
            \{m\in \Z \mid 1-r_1\mu_n-y\le m\le (r_1+Tr_2)\mu_n-y\} & \text{if $b\neq 1$}, \\
            \{m\in \Z \mid 1-r_1\mu_n-y\le m\le (r_1+Tr_2)\mu_n+1-y\}\setminus\{1\} & \text{if $b= 1$}.
        \end{cases}
    \]
    \begin{proof}[Proof of Lemma~\ref{lemma: asymptotics for Q_b''}]
        We start with the behaviour of $G_{\underline{\ell}}(y)$. Fix $y\in \{0,\dots, r_2\mu_n\}$ and $\theta=\frac{a}{b}\in \mathcal{Z}_B$ with $\gcd(a,b)=1$. Assume $b\neq 1$. Consider the power sum 
        \[
            P_k = \sum_{j\in J^{(y)}_{b,a}} (-a+bj)^k.
        \]
        Approximating $P_k$ with the integral of $t\mapsto (-a+bt)^k$ over a suitable interval, we obtain
        \[
            P_k \sim \frac{b^k}{k+1} \mu_n^{k+1} \left( \left(r_1+Tr_2-\frac{y}{\mu_n}\right)^{k+1}-(-1)^{k+1}\left(r_1+\frac{y}{\mu_n}\right)^{k+1} \right).
        \]
        In particular $P_1^k \sim c(y) \mu_n^{2k}$ for some constant $c(y)>0$. Note that $0\le y/\mu_n\le r_2$, so it does not disrupt the asymptotics. It follows that $P_k=o(P_1^k)$ for all $k\ge 2$. \\
        Note that $g_{\ell,b,a}(y)$ is the $\ell$-th symmetric polynomial in $-a+bj$ with $j\in J^{(y)}_{b,a}$, hence
        \[
            g_{\ell,b,a}(y)= (-1)^\ell \sum_{\substack{m_1+2m_2+\dots \ell m_\ell=\ell \\ m_1\ge 0}} \prod_{i=1}^\ell \frac{(-P_i)^{m_i}}{m_i! i^{m_i}}.
        \]
        Since $P_k=o(P_1^k)$ for $k\ge 2$, we deduce that 
        \[
            g_{\ell,b,a}(y) = b^\ell C(y,\ell)
            \mu_n^{2\ell}\left(1+O\left(\frac{1}{\mu_n}\right)\right),
        \]
        where
        \[
            C(y,\ell) = \frac{1}{2 ^\ell \ell!}\left( \left(r_1+Tr_2-\frac{y}{\mu_n}\right)^{2}-\left(r_1+\frac{y}{\mu_n}\right)^{2} \right)^\ell.
        \]
        The case $b=1$ is analogous up to a minor redefinition of some constants. Note that, if $T\ge 3$, then $C(y,\ell)>0$. If $T=2$ instead, $C(y,\ell)$ vanishes when $y=r_2\mu_n$. This particular case will not affect the asymptotic estimates that follow. \\
        As a result, given $\underline{\ell}=(\ell_{\beta,\alpha})_{\frac{\alpha}{\beta}\in \mathcal{Z}_B}$, we have
        \[
            G_{\underline{\ell}}(y) \sim \mu_n ^{2\sum_{\frac{\alpha}{\beta}}\ell_{\beta,\alpha}}\prod_{\frac{\alpha}{\beta}\in \mathcal{Z}_B} \beta^{\ell_{\beta,\alpha}}C(y,\ell_{\beta,\alpha}). 
        \]
        We move to $Q_i(y)$. The indices $\ell_\nu'$ in \eqref{eq: Lambda_i} are bounded above by $s+1-i$, so the sum over $\nu$ in \eqref{eq: Q_i(y)} has absolute value at most $\mu_n$, up to a constant. Given the asymptotic behaviour of $G_{\underline{\ell}}(y)$, it follows that $Q_i(y)$ is dominated by the tuples in $\Lambda_i$ of the form \eqref{eq: Lambda_i} such that $\ell_\nu'=0$ for all $\nu$. So,
        \[
            Q_i(y)=C_i'(y)\mu_n ^{2(s+1-i)+1}\left(1+O\left(\frac{1}{\mu_n}\right) \right)
        \]
        where
        \[
            C_i'(y)=r_2\sum_{\ell_{\beta,\alpha}}\prod_{\frac{\alpha}{\beta}\in \mathcal{Z}_B} \beta^{\ell_{\beta,\alpha}}C(y,\ell_{\beta,\alpha})
        \]
        and the sum runs over $\sum_{\frac{\alpha}{\beta}\in \mathcal{Z}_B} \ell_{\beta,\alpha}=s+1-i$. Note that $C_i'(y)>0$ because $C(y,\ell)>0$ for all $\ell$, when $T\ge 3$. In case $T=2$ one of the $C(y,\ell)$'s will vanish, but the overall sum remains positive. \\
        Fix $b\in \Psi_B$. Inserting the asymptotic estimate of $Q_i(y)$ in the definition of $Q_b''$, we get
        \[
            Q_b''= 
            \sum_{i=1}^s
            b^{i}
            \mu_n ^{2(s+1-i)}
            \sum_{y=0}^{r_2\mu_n} \left(1+O\left(\frac{1}{\mu_n}\right)\right) C_i'(y)\binom{r_2\mu_n}{y}^{s+1}\sum_{j=1}^{Tby} \frac{\chi(-j)}{j^i}.
        \]
        Note that the function implicit in the $O$-term depends on $i$ and $y$. We now split this sum into 
        \begin{enumerate}
        \item an error term:
        \[
            E= \sum_{i=1}^s
        b^{i}
        \mu_n ^{2(s+1-i)+1}
        \sum_{y=0}^{r_2\mu_n}  O\left(\frac{1}{\mu_n}\right) C_i'(y)\binom{r_2\mu_n}{y}^{s+1}\sum_{j=1}^{Tby} \frac{\chi(-j)}{j^i}.
        \]
        \item the terms for $i\ge 2$. Define
            \[
                S_0=\sum_{i=2}^s
        b^{i}
        \mu_n ^{2(s+1-i)+1}
        \sum_{y=0}^{r_2\mu_n}  C_i'(y)\binom{r_2\mu_n}{y}^{s+1}\sum_{j=1}^{Tby} \frac{\chi(-j)}{j^i}.
            \]
        \item the term for $i=1$ and $y\le \lfloor \log \mu_n \rfloor $. Define 
        \[
            S_1=b\mu_n^{2s+1}\sum_{y=0}^{\lfloor\log\mu_n\rfloor}  C_1'(y)\binom{r_2\mu_n}{y}^{s+1}\sum_{j=1}^{Tby} \frac{\chi(-j)}{j}.
        \]
        \item the term for $i=1$ and $y\ge \lfloor \log \mu_n \rfloor+1$. Define 
        \[
            S_2=b\mu_n^{2s+1}\sum_{y=\lfloor\log\mu_n\rfloor+1}^{r_2\mu_n}  C_1'(y)\binom{r_2\mu_n}{y}^{s+1}\sum_{j=1}^{Tby} \frac{\chi(-j)}{j}.
        \]
    \end{enumerate}
    We will prove that $S_2$ dominates the other terms in the list. Let us start with $S_0$. Note that for $i\ge 2$
    \[
        \left| \sum_{j=1}^{Tby} \frac{\chi(-j)}{j^i} \right| \le \zeta(i)
    \]
    for all $y\ge 1$. Thus, the $i$-th term in $S_0$ is bounded by
    \[
        b^i \zeta(i) \mu_n^{2(s+1-i)+1} \sum_{y=0}^{r_2\mu_n}  C_i'(y)\binom{r_2\mu_n}{y}^{s+1}.
    \]
    All $C_i'(y)$ are bounded above by a constant, and the leading binomial coefficient is the middle one. Since the number of terms in the sum is $r_2\mu_n+1$, we have 
    \begin{align*}
        |S_0| \le c_1\mu_n^{2s} 2^{(s+1)r_2\mu_n}\sqrt{\frac{2^{s+1}}{(\pi r_2\mu_n)^{s+1}}}
    \end{align*}
    for some constant $c_1>0$. \\
    We now consider $S_1$. Note that for $y\le \lfloor\log\mu_n\rfloor$ we have
    \[
        \binom{r_2\mu_n}{y} \le \binom{r_2\mu_n}{\lfloor\log\mu_n\rfloor} \le \frac{(r_2\mu_n)^{\log\mu_n}}{(\lfloor\log\mu_n\rfloor)!}.
    \]
    Hence,
    \[
        |S_1|\le c_2\log\mu_n(\log\mu_n+1) \left( \frac{(r_2\mu_n)^{\log\mu_n}}{(\lfloor\log\mu_n\rfloor)!} \right)^{s+1}\mu_n^{2s+1}
    \]
    for some constant $c_2>0$. \\
    Let us move to $S_2$. For $n\to \infty$, we have $y\to \infty$ in the second sum in the definition of $S_2$, so 
    \[
        \sum_{j=1}^{Tby}\frac{\chi(-j)}{j} = (-1)^{\sigma(\chi)}L(1,\chi)+O\left(\frac{1}{Tby}\right).
    \]
    Observe that $L(1,\chi)\neq 0$ by Dirichlet's theorem. We deduce that
    \[
        S_2= b\mu_n^{2s+1}\left( (-1)^{\sigma(\chi)}L(1,\chi) + O\left(\frac{1}{\log \mu_n} \right)\right)
        \sum_{y=\lfloor\log\mu_n\rfloor+1}^{r_2\mu_n}  C_1'(y)\binom{r_2\mu_n}{y}^{s+1}
    \]
    The last sum is positive, as all $C_1'(y)$'s are positive. Moreover, there is a constant $c_3>0$ such that $C_1'(y)>c_3$ for all $y\in \{0,\dots, r_2\mu_n\}$. Thus,
    \[
        \sum_{y=\lfloor\log\mu_n\rfloor+1}^{r_2\mu_n}  C_1'(y)\binom{r_2\mu_n}{y}^{s+1}\ge  
        c_32^{(s+1)r_2\mu_n}\sqrt{\frac{2^{s+1}}{(\pi r_2\mu_n)^{s+1}}}.
    \]
    As $n\to \infty$, the exponential growth of $S_2$ coming from the central binomial coefficient dominates over $S_1$, so $S_1=o(S_2)$. On the other hand, this growth factor is the same as the one appearing in $S_0$. However, $S_0$ contains at most a factor $\mu_n^{2s-1}$, while $S_2$ at least a factor $\mu_n^{2s}$. This implies that $S_0=o(S_2)$. Finally, the error term $E$ can be estimated similarly to $S_2$. After removing the dependence on $y$, it follows that $|E|=O(\mu_n^{-1})|S_2|$, so $E=o(S_2)$. As a result, for $n\to\infty$,
    \[
        Q_b''\sim 
        (-1)^{\sigma(\chi)}bL(1,\chi)
        \mu_n ^{2s+1}
        \sum_{y=\lfloor \log\mu_n\rfloor +1}^{r_2\mu_n}  C_1'(y)\binom{r_2\mu_n}{y}^{s+1}.
    \]
    Since $L(1,\chi)\neq 0$, this shows that $Q_b''\neq 0$ for all $b\in \Psi_B$ and all sufficiently large $n$. Moreover, the dependence on $b$ is dictated only by the factor $b$. Thus,
    \[
        \lim_{n\to \infty} \frac{Q_b''}{Q_1''}=b.
    \] 
    This concludes the proof of Lemma~\ref{lemma: asymptotics for Q_b''}.
    \end{proof}

%% file: 5_Analytic_estimates.tex
The goal of this section is to study the growth behaviour of the linear forms $r_{n,\theta}$ for $\theta\in \mathcal{Z}_B$. Since we rely on the irrationality criterion of Lemma~\ref{lemma: irrationality criterion}, we will only need an upper bound for their absolute value. The computations closely follow \cite{Lai-Yu} with some minor modifications. \\
For $j=0,1,\ldots,T-1$ and $\theta\in\mathcal Z_B$, define
\[
r_{n,j,\theta}:=\sum_{k=0}^{\infty}\hat{R}_n(kT+j+\theta),
\qquad
\hat{R}_n(t):=R_n(t+rn).
\]
\begin{proposition}
\label{analysis-lem}
We have
\[
\lim_{n\to\infty}(r_{n,j,\theta})^{\frac1n}=g(x_0),
\]
where
\[
g(X)=A_{1}(B)A_2(B)\mathrm{den}(r)^{(2r+T)|\mathcal{Z}_B|}(XT+2r+T)^{(2r+T)|\mathcal{Z}_B|}
\left(\frac{(X+\frac{r}{T})^{\frac{r}{T}}}{(X+1+\frac{r}{T})^{\frac{r}{T}+1}}\right)^{s+1}\frac{1}{T^{s+1}},
\]
and $x_0$ is the unique positive real solution of the equation $f(x)=1$, with
\[
f(X)=\left(\frac{XT +2 r+T}{X T}\right)^{|\mathcal{Z}_B| T}\left(\frac{XT+r}{XT+T+r}\right)^{s+1}.
\]
\end{proposition}
Before the proof, note the following immediate consequence for our linear forms:
\begin{corollary}
    For all $\theta\in \mathcal{Z}_B$, the following inequality holds:
    \[
        \limsup_{n\to\infty}|r_{n,\theta}|^{\frac1n}\leq g(x_0).
    \]
\end{corollary}
We first collect some information on the function $f$ in the statement of Proposition~\ref{analysis-lem}.
\begin{lemma}\label{prop-function}
Let $f(x)$ be as in Proposition~\ref{analysis-lem}, defined on $(0,\infty)$. Then:
\begin{enumerate}
\item[(1)] there exists a unique $x_0\in (0,\infty)$ such that $f(x_0)=1$, $f(x)>1$ on $(0,x_0)$ and $f(x)<1$ on $(x_0,\infty)$. Moreover,
\[
x_0<\frac{r(r+T)|\mathcal{Z}_B|}{s+1-(2r+T)|\mathcal{Z}_B|}.
\]
\item[(2)] if we fix $r$ and $T$ and assume that $B=c \left(\frac{s}{\log s}\right)^{1/2}$ for some constant $c>0$, then, as $s\to\infty$, we have
\[
\lim_{s\to\infty}g(x_0)^{1/(s+1)}=\frac{1}{T}\exp\left(\frac{\zeta(2)\zeta(3)}{4 \zeta(6)}c^2(2r+T)\right)\frac{(\frac{r}{T})^{\frac{r}{T}}}{(1+\frac{r}{T})^{1+\frac{r}{T}}}.
\]
\end{enumerate}
\end{lemma}

\begin{proof} 
    The result follows as in \cite[Proposition 4.2]{Lai-Yu}.
\end{proof}

In the rest of this section we prove Proposition~\ref{analysis-lem}.
\begin{proof}[Proof of Proposition~\ref{analysis-lem}]
For any $\theta\in\mathcal Z_B$, by the construction of the zeros of $R_n$ we have
$R_n(t+\theta)=0$ for $t=1,\ldots,rn-1$. We define the shifted rational function
\[
\hat{R}_n(t)=R_n(t+rn).
\]
Note that
\begin{equation}\label{eq:estimate series}
\begin{aligned}
r_{n,\theta} =\sum_{j=0}^{T-1}\chi_\theta(j)r_{n,j,\theta}.
\end{aligned}
\end{equation}
Thus the quantities $r_{n,j,\theta}$ are positive sums, while $r_{n,\theta}$ may have cancellation because of the character. For $t>0$, from the definition of $R_n$ we have
\begin{align*}
\hat{R}_n(t)
&=A_1(B)^nA_2(B)^n
\frac{(n+1)!^{s+1}}
{\bigl((n/\mathrm{den}(r))!\bigr)^{\mathrm{den}(r)(2r+T)|\mathcal Z_B|}}
\frac{t\prod_{\gamma\in\mathcal Z_B}\prod_{i=0}^{(2r+T)n-1}(t+i+\gamma)}
{\prod_{i=0}^{n}(t+rn+Ti)^{s+1}}.
\end{align*}
Consequently, for $j=0,1,\ldots,T-1$,
\begin{align*}
\frac{\hat{R}_n((k+1)T+j+\theta)}{\hat{R}_n(kT+j+\theta)}
&=\frac{(k+1)T+j+\theta}{kT+j+\theta}
\prod_{\gamma\in\mathcal Z_B}\prod_{\ell=0}^{T-1}
\frac{kT+j+\theta+(2r+T)n+\gamma+\ell}
{kT+j+\theta+\gamma+\ell}\\
&\qquad\times
\left(
\frac{kT+j+\theta+rn}
{kT+j+\theta+rn+T(n+1)}
\right)^{s+1}.
\end{align*}
If $k\sim\kappa n$ for a fixed $\kappa>0$, then the first factor tends to $1$, and therefore
\[
\frac{\hat{R}_n((k+1)T+j+\theta)}{\hat{R}_n(kT+j+\theta)}
\longrightarrow
f(\kappa),
\]
where
\[
f(X)=\left(\frac{XT+2r+T}{XT}\right)^{|\mathcal Z_B|T}
\left(\frac{XT+r}{XT+T+r}\right)^{s+1}.
\]
For every $t>0$, we have
\begin{align}
\hat{R}_n(t)
&=A_1(B)^nA_2(B)^n
\frac{(n+1)!^{s+1}}
{\bigl((n/\mathrm{den}(r))!\bigr)^{\mathrm{den}(r)(2r+T)|\mathcal Z_B|}}
\frac{t}{T^{(n+1)(s+1)}}\notag\\
&\quad\times
\left(\prod_{\gamma\in\mathcal Z_B}
\frac{\Gamma(t+(2r+T)n+\gamma)}{\Gamma(t+\gamma)}\right)
\left(
\frac{\Gamma(\frac{t+rn}{T})}
{\Gamma(\frac{t+rn}{T}+n+1)}
\right)^{s+1}.
\label{eq:gamma-expression}
\end{align}
Denote by $k_0(n)$ the integer part of $x_0n$. Taking
$t=k_0(n)T+j+\theta$ in \eqref{eq:gamma-expression}, we obtain
\begin{align*}
\hat{R}_n(k_0(n)T+j+\theta)
&=A_1(B)^nA_2(B)^n
\frac{(n+1)!^{s+1}}
{\bigl((n/\mathrm{den}(r))!\bigr)^{\mathrm{den}(r)(2r+T)|\mathcal Z_B|}}
\frac{k_0(n)T+j+\theta}{T^{(n+1)(s+1)}}\\
&\quad\times
\prod_{\gamma\in\mathcal Z_B}
\frac{\Gamma(k_0(n)T+j+\theta+(2r+T)n+\gamma)}
{\Gamma(k_0(n)T+j+\theta+\gamma)}\\
&\quad\times
\left(
\frac{\Gamma(k_0(n)+\frac{j+\theta}{T}+\frac{rn}{T})}
{\Gamma(k_0(n)+\frac{j+\theta}{T}+\frac{rn}{T}+n+1)}
\right)^{s+1}.
\end{align*}
By Stirling's formula, we get
\begin{align*}
&\hat{R}_n(k_0(n)T+j+\theta)^{1/n}\\
&\quad\sim
\frac{A_1(B)A_2(B)\mathrm{den}(r)^{(2r+T)|\mathcal Z_B|}}{T^{s+1}}
\left(
\frac{(x_0T+2r+T)^{x_0T+2r+T}}
{(x_0T)^{x_0T}}
\right)^{|\mathcal Z_B|}\\
&\qquad\qquad\times
\left(
\frac{(x_0+\frac rT)^{x_0+\frac rT}}
{(x_0+1+\frac rT)^{x_0+1+\frac rT}}
\right)^{s+1}\\
&\quad=f(x_0)^{x_0}g(x_0)=g(x_0).
\end{align*}
We now show that the asymptotic behaviour of $r_{n,j,\theta}$ for $n\in \mathcal{N}$, $n\to \infty$ is governed by the terms $\hat R_n(kT+j+\theta)$ with $k$ close to $x_0n$. We take $\varepsilon>0$ sufficiently small so that $x_0-3\varepsilon>0$ and put
\[
b(\varepsilon)=\max\left(f(x_0+\varepsilon),\frac1{f(x_0-\varepsilon)}\right)<1.
\]
Since $f$ is continuous and increases to $1$ on $[x_1,\infty)$, there exists $A(\varepsilon)>x_1$ such that
\[
f(A(\varepsilon))=b(\varepsilon).
\]
Moreover, $A(\varepsilon)\to\infty$ as $\varepsilon\to0^+$. Further more, $f(x)\geq \frac1{b(\varepsilon)}$ for $0<x\leq x_0-\varepsilon$, while $f(x)\leq b(\varepsilon)$ for $x_0+\varepsilon\leq x\leq A(\varepsilon)$. \\
We split the sum
\[
r_{n,j,\theta}=\sum_{k=0}^{\infty}\hat R_n(kT+j+\theta)
\]
into the following four parts:
\begin{align*}
\sum_{0\leq k\leq\lfloor(x_0-3\varepsilon)n\rfloor}
+\sum_{\lfloor(x_0-3\varepsilon)n\rfloor<k<\lfloor(x_0+3\varepsilon)n\rfloor}
+\sum_{\lfloor(x_0+3\varepsilon)n\rfloor\leq k\leq\lfloor(A(\varepsilon)-\varepsilon)n\rfloor}
+\sum_{k\geq\lceil(A(\varepsilon)-\varepsilon)n\rceil}.
\end{align*}

We first estimate the left part. From the exact quotient formula and the strict inequality
\[
f(x)>\frac1{b(\varepsilon)}
\qquad(0<x\leq x_0-2\varepsilon),
\]
we obtain, for all sufficiently large $n$ and $0\leq k\leq \lfloor(x_0-2\varepsilon)n\rfloor-1$
\[
\frac{\hat R_n((k+1)T+j+\theta)}{\hat R_n(kT+j+\theta)}
\geq\frac1{b(\varepsilon)}.
\]
Put
\[
k_1=\lfloor(x_0-3\varepsilon)n\rfloor,
\qquad
k_2=\lfloor(x_0-2\varepsilon)n\rfloor.
\]
Then
\begin{align}
\label{eq: estimate for part1}
\sum_{0\leq k\leq k_1}\hat R_n(kT+j+\theta)
\leq \frac{b(\varepsilon)^{k_2-k_1}}{1-b(\varepsilon)}
\hat R_n(k_2T+j+\theta) \leq \varepsilon\hat R_n(k_2T+j+\theta).
\end{align}
To estimate the third sum, one may argue similarly and obtain
\begin{align}
\sum_{\lfloor(x_0+3\varepsilon)n\rfloor\leq k\leq\lfloor(A(\varepsilon)-\varepsilon)n\rfloor}
\hat R_n(kT+j+\theta) \le 
\varepsilon\hat R_n(k_3T+j+\theta).
\label{eq:estimate for part2}
\end{align}
We move to the tail of the series. For $k\ge (A(\varepsilon)-\varepsilon)n$, we have
\[
kT+j+\theta+(2r+T)n+\gamma\leq2kT
\]
for every $\gamma\in\mathcal Z_B$. Therefore, by \eqref{eq:gamma-expression}, we get
\begin{align*}
\hat R_n(kT+j+\theta)
&\leq
\frac{A_1(B)^nA_2(B)^n(n+1)!^{s+1}}
{\bigl((n/\mathrm{den}(r))!\bigr)^{\mathrm{den}(r)(2r+T)|\mathcal Z_B|}}
\frac{(2kT)^{(2r+T)|\mathcal Z_B|n+1}}
{(kT)^{(s+1)(n+1)}}.
\end{align*}
Using Stirling's formula, we have
\begin{align*}
\hat R_n(kT+j+\theta)
&\leq n^{O(1)}
\left(
\frac{A_1(B)A_2(B)\mathrm{den}(r)^{(2r+T)|\mathcal Z_B|}
2^{(2r+T)|\mathcal Z_B|}}{\exp(s+1-(2r+T)|\mathcal Z_B|)}
\right)^n\\
&\quad\times
\frac{n^{(s+1-(2r+T)|\mathcal Z_B|)n}}{(kT)^{(s+1-(2r+T)|\mathcal Z_B|)n-s}}.
\end{align*}
Since
\[
s+1-(2r+T)|\mathcal Z_B|>0,
\]
we infer that 
\begin{align*}
\limsup_{n\to\infty}
\left(
\sum_{k\geq (A(\epsilon)-\epsilon)n}\hat R_n(kT+j+\theta)
\right)^{1/n}
\leq
\frac{
A_1(B)A_2(B)\mathrm{den}(r)^{(2r+T)|\mathcal Z_B|}
2^{(2r+T)|\mathcal Z_B|}
}
{e^{s+1-(2r+T)|\mathcal Z_B|}(A(\epsilon)-\epsilon)T)^{s+1-(2r+T)|\mathcal Z_B|}}.
\end{align*}
The right-hand side tends to $0$ as $\epsilon\to 0$, so we may assume it to be strictly smaller than $g(x_0)/2$. Hence, for all sufficiently large $n$,
\begin{equation}\label{eq:tailpart-estimate}
\sum_{k\geq\lceil(A(\varepsilon)-\varepsilon)n\rceil}
\hat R_n(kT+j+\theta)
\leq\left(\frac{g(x_0)}{2}\right)^n.
\end{equation}
On the other hand, from the already established asymptotic behaviour around $x_0n$,
\[
\hat R_n(k_0(n)T+j+\theta)\geq\left(\frac{3g(x_0)}4\right)^n.
\]
Combining this with \eqref{eq:tailpart-estimate}, we get
\[
\sum_{k\geq\lceil(A(\varepsilon)-\varepsilon)n\rceil}
\hat R_n(kT+j+\theta)
\leq\left(\frac23\right)^n
\hat R_n(k_0(n)T+j+\theta)
\leq \varepsilon r_{n,j,\theta}
\]
for all sufficiently large $n$. Similarly, the quantities in \eqref{eq: estimate for part1} and \eqref{eq:estimate for part2} are bounded by $\varepsilon r_{n,j,\theta}$. Thus, the three parts away from $x_0n$ contribute at most $3\varepsilon r_{n,j,\theta}$, and therefore
\begin{equation}\label{eq:comparison}
(1-3\varepsilon)r_{n,j,\theta}
\leq
\sum_{(x_0-3\varepsilon)n\leq k\leq(x_0+3\varepsilon)n}
\hat R_n(kT+j+\theta)
\leq r_{n,j,\theta}.
\end{equation}
For all sufficiently large $n$ and all $k$ satisfying
\[
(x_0-3\varepsilon)n\leq k\leq(x_0+3\varepsilon)n,
\]
there is a positive function $h(\epsilon)$ going to zero as $\epsilon\to 0$ such that
\begin{equation}\label{eq:limit-conclusion}
(g(x_0)-h(\epsilon))^n
\leq\hat R_n(kT+j+\theta)
\leq(g(x_0)+h(\epsilon))^n.
\end{equation}
Using \eqref{eq:comparison} and the upper bound in \eqref{eq:limit-conclusion}, we obtain
\begin{align*}
r_{n,j,\theta}
&\leq \frac1{1-3\varepsilon}
\sum_{(x_0-3\varepsilon)n\leq k\leq(x_0+3\varepsilon)n}
\hat R_n(kT+j+\theta)\\
&\leq \frac{7\varepsilon n}{1-3\varepsilon}
(g(x_0)+h(\epsilon))^n.
\end{align*}
Hence, by the lower bound in \eqref{eq:limit-conclusion},
\[
r_{n,j,\theta}
\geq5\varepsilon n(g(x_0)-h(\epsilon))^n.
\]
Taking $n$-th roots and then letting $n\to\infty$, we obtain
\[
g(x_0)-h(\epsilon)
\leq\liminf_{n\to\infty}(r_{n,j,\theta})^{1/n}
\leq\limsup_{n\to\infty}(r_{n,j,\theta})^{1/n}
\leq g(x_0)+h(\epsilon).
\]
Finally, letting $\varepsilon\to 0$ gives
\[
\lim_{n\to\infty}(r_{n,j,\theta})^{1/n}=g(x_0)
\]
for every $j=0,1,\ldots,T-1$ and every $\theta\in\mathcal Z_B$. This completes the proof.
\end{proof}

%% file: 6_Proof_of_the_main_result.tex
    In this section we conclude the proof of Theorem~\ref{theorem: main}, which combines the elimination technique of \cite{Fischler-Sprang-Zudilin} with the irrationality criterion of \cite{Lai-Sprang}. \\
    Let $\psi=\#\Psi_B-1$. Suppose that there are at most $\psi$ irrational numbers among 
    \[
        \{ L(i,\chi) \mid 2\le i\le s, \; i\equiv \sigma(\chi)+1\modulo{2}\}.
    \]
    We write $L(i_1,\chi),\dots, L(i_\psi,\chi)$ for the irrational elements in the above set. We also write $b_0,\dots, b_\psi$ for the elements in $\Psi_B$. We have seen in Proposition~\ref{proposition: det ist not zero} that the following $(\psi+1)\times(\psi+1)$ matrix is invertible:
    \[
        M=
        \begin{pmatrix}
            Q_{b_0}'' & b_0^{i_1} & \dots & b_0^{i_\psi} \\
            Q_{b_1}'' & b_1^{i_1} & \dots & b_1^{i_\psi} \\
            \vdots & \vdots & \ddots & \vdots \\
            Q_{b_\psi}'' & b_\psi^{i_1} & \dots & b_\psi^{i_\psi} \\
        \end{pmatrix}
    \]
    Hence, there are integers $(w_b)_{b\in \Psi_B}$ such that for all $j=1,\dots,\psi$
    \[
        \sum_{b\in \Psi_B} w_bQ_b'''\neq 0, \qquad \sum_{b\in \Psi_B} w_b b^{i_j}=0.
    \]
    For this choice of $w_b$'s, the first condition means precisely $\widetilde{Q}\neq 0$. 
    In particular, by Corollary~\ref{corollary: p-adic valuation of Q tilde},
    \[
        v_{p_n}(\rho_0'')=-(s+1)
    \]
    On the other hand, for $i\ge 2$
    \[
        v_{p_n}(\rho_i'')\ge -(s+1-i)
    \]
    Thus, $\rho_0''$ has strictly smaller $p_n$-adic valuation than each $\rho_i''$. \\
    Consider the linear forms $r_n''$ defined in Section~\ref{section: linear forms} obtained for this choice of $w_b$'s. Given $\epsilon>0$, by Proposition~\ref{analysis-lem}, for all $n\in \mathcal{N}$ sufficiently large we have 
    \[
        |r_n''|\le (1+\epsilon)\left( \sum_{b\in \Psi_B} |w_b| b \right)g(x_0)^n.
    \]
    Let $D$ be a common denominator for all rational $L(i,\chi)$ with $2\le i \le s$. Then
    \[
        |D\delta_n^{s+1}r_n''|\le DT^{s+1}(1+\epsilon)^2\left( \sum_{b\in \Psi_B} |w_b| b \right)(e^{s+1}g(x_0))^n.
    \]
    The linear forms $D\delta_n^{s+1}r_n''$ have integer coefficients. Moreover, the coefficient of $1$ has strictly smaller $p_n$-adic valuation than the other ones. By Lemma~\ref{lemma: irrationality criterion}, we deduce the following fact:
    \begin{lemma}
    \label{lemma: final lemma}
        Suppose that
        \[
            e^{s+1}g(x_0)<1.
        \]
        Then there are at least $\#\Psi_B$ irrational numbers among
        \[
            \{ L(i,\chi) \mid 2\le i\le s, \; i\equiv \sigma(\chi)\modulo{2}\}.
        \]
    \end{lemma}
    To conclude the proof of Theorem~\ref{theorem: main} we explain how to choose the parameters $B$ and $r$ so that the condition in Lemma~\ref{lemma: final lemma} is verified. We start with 
    \begin{equation}
    \label{eq: definition of B}
        B=\widetilde{c}\, \sqrt{\frac{s}{\log s}}
    \end{equation}
    for some constant $\widetilde{c}>0$ to be determined. From the analytic study of the linear forms we obtain
    \[
        \lim_{s\to\infty}g(x_0)^{1/(s+1)}=\frac{1}{T}\exp\left(\frac{\zeta(2)\zeta(3)}{4 \zeta(6)}\widetilde{c}^2(2r+T)\right)\frac{(\frac{r}{T})^{\frac{r}{T}}}{(1+\frac{r}{T})^{1+\frac{r}{T}}}.
    \]
    For this quantity to be smaller than $e^{-1}$, we need to impose
    \[
        \widetilde{c}<\sqrt{\frac{4\zeta(6)}{T\zeta(2)\zeta(3)}\frac{\log T +\left(1+\frac{r}{T}\right)\log\left(1+\frac{r}{T}\right)-\frac{r}{T}\log\left(\frac{r}{T}\right)-1}{2\frac{r}{T}+1}}
    \]
    The maximum of the function
    \[
        x\mapsto \frac{\log T +\left(1+x\right)\log\left(1+x\right)-x\log\left(x\right)-1}{2x+1}
    \]
    is attained by
    \[
        \widetilde{x}=\frac{\sqrt{4e^2+T}-T}{2T}.
    \]
    Thus, the maximum of the above expression in terms of $r$ is reached at
    \[
        \widetilde{r}=\frac{\sqrt{4e^2+T}-T}{2},
    \]
    with maximum value
    \[
        \frac{ 1}{2}\log \left(\frac{\sqrt{4e^2+T^2}+T}{\sqrt{4e^2+T^2}-T} \right)
    \]
    Choose
    \[
        c(T)=\sqrt{\frac{2\zeta(2)\zeta(3)}{\zeta(6)T}\log \left(\frac{\sqrt{4e^2+T^2}+T}{\sqrt{4e^2+T^2}-T} \right)}.
    \]
    Fix $k>0$ such that
    \[
        k>1+\frac{\zeta(6)c(T)}{\zeta(2)\zeta(3)}.
    \]
    Given $\epsilon>0$, choose a rational number $r$, depending on $x_0$, sufficiently close to $\widetilde{r}$ such that 
    \[
        \frac{\zeta(6)}{\zeta(2)\zeta(3)}\left(c(T)-\frac{\epsilon}{k}\right) < \sqrt{\frac{4\zeta(6)}{T\zeta(2)\zeta(3)}\frac{\log T +\left(1+\frac{r}{T}\right)\log\left(1+\frac{r}{T}\right)-\frac{r}{T}\log\left(\frac{r}{T}\right)-1}{\frac{2r}{T}+1}}.
    \]
    In \eqref{eq: definition of B} we set
    \[
        \widetilde{c}=\frac{\zeta(6)}{\zeta(2)\zeta(3)}\left(c(T)-\frac{\epsilon}{k}\right).
    \]
    Then there is a sufficiently large $s$, depending on $r$ and $\epsilon$, so that $g(x_0)<e^{-(s+1)}$ and 
    \[
        |\Psi_B|>\left(\frac{\zeta(2)\zeta(3)}{\zeta(6)}-\frac{\epsilon}{k}\right) B. 
    \] 
    Thus, by Lemma~\ref{lemma: final lemma} the number of irrationals $L$-values 
    \[
        L(j,\chi), \qquad 2\le j\le s, \; j\equiv \sigma(\chi) \modulo{2}
    \]
    is at least
    \[
        |\Psi_B|>(c(T)-\epsilon)\sqrt{\frac{s}{\log s}},
    \]
    which proves Theorem~\ref{theorem: main}.